\documentclass[12pt]{amsart}

\usepackage{mathrsfs}

\usepackage{amsthm}

\usepackage{comment}

\usepackage{amsfonts,amsmath,amssymb,amsxtra,anysize,times,tikz, float, latexsym}

\usetikzlibrary{arrows,decorations.markings}

\usepackage{tabularx, booktabs}
\newcolumntype{Y}{>{\centering\arraybackslash}X}

\usepackage{subfigure}
\usepackage{xcolor}

\definecolor{mygreen}{rgb}{0.0, 0.5, 0.0}

\author{Eugene Gorsky}
\address{University of California at Davis, Davis, California, US}
\address{International Laboratory of Representation Theory and Mathematical Physics, NRU-HSE, Moscow, Russia}
\email{egorskiy@math.ucdavis.edu}
\author{Mikhail Mazin }
\address{Kansas State University, Manhattan, Kansas, US}
\email{mmazin@math.ksu.edu}
\author{ Monica Vazirani} 
\address{University of California at Davis, Davis, California, US}
\email{vazirani@math.ucdavis.edu}
\title{Rational Dyck Paths in the Non Relatively Prime Case}


\newtheorem{lemma}{Lemma}[section]
\newtheorem{theorem}[lemma]{Theorem}
\newtheorem{conjecture}[lemma]{Conjecture}
\newtheorem{corollary}[lemma]{Corollary}
\newtheorem{proposition}[lemma]{Proposition}
\theoremstyle{definition} 
\newtheorem{example}[lemma]{Example}
\newtheorem{remark}[lemma]{Remark}
\newtheorem{definition}[lemma]{Definition}
 
\DeclareMathOperator{\col}{col}
\DeclareMathOperator{\Rspan}{span}
\DeclareMathOperator{\Perm}{Perm}
\DeclareMathOperator{\An}{\mathcal{A}} 
\DeclareMathOperator{\Core}{Core}

\newcommand{\dinv}{\mathtt{dinv}}  
\newcommand{\arm}{\mathtt{arm}}  
\newcommand{\leg}{\mathtt{leg}}  
\newcommand{\area}{\mathtt{area}}  
\newcommand{\coarea}{\mathtt{co-area}}  
\newcommand{\gap}{\mathtt{gap}}  
\newcommand{\bepsilon}{\bar{\epsilon}}   

\newcommand{\miasing}{minimal integral acceptable shifting}
\newcommand{\Sym}{\mathcal{S}} 

\newcommand{\Sd}{\Sym_d}

\newcommand{\BZ}{\mathbb{Z}}
\newcommand{\BR}{\mathbb{R}}
\newcommand{\ZZ}{\mathbb{Z}_{\ge 0}}

\newcommand{\CD}{\mathcal{D}}
\newcommand{\CG}{\mathcal{G}}

\newcommand{\rect}{R}

\newcommand{\RNM}{\rect_{N,M}}

\newcommand{\M}{\mathbf{M}}  
\newcommand{\Mi}[2]{\M_{#1,#2}}
\newcommand{\Mnm}{\Mi{n}{m}}
\newcommand{\MNM}{\Mi{N}{M}}

\newcommand{\ver}{\mathtt{v}}
\newcommand{\hor}{\mathtt{h}}
\newcommand{\bver}{\mathbf{{v}}}
\newcommand{\bhor}{\mathbf{{h}}}
\newcommand{\verc}{\color{red} \mathtt{v}} 
\newcommand{\horc}{\color{blue}{\mathtt{h}}}

\newcommand{\Y}{Y}
\newcommand{\Yi}[2]{\Y_{{#1,#2}}}
\newcommand{\Ynm}{\Yi{n}{m}}
\newcommand{\YNM}{\Yi{N}{M}}
\newcommand{\nm}{(n,m)} 
\newcommand{\NM}{(N,M)} 
\newcommand{\rank}{\mathtt{rank}} 
\newcommand{\Z}{\mathbb{Z}}

\usepackage{todonotes}

\begin{document}
\begin{abstract}

We study the relationship between rational slope Dyck paths  and invariant subsets of $\BZ,$ extending the work of the first two authors in the relatively prime case. We also find a bijection between $(dn,dm)$--Dyck paths and $d$-tuples of $(n,m)$-Dyck paths endowed with certain gluing data. These are the first steps towards understanding the relationship between rational slope Catalan combinatorics
and the geometry of affine Springer fibers and knot invariants
in the non relatively prime case.

\end{abstract}
\date{\today}
\maketitle

\section{Introduction}
\label{sec-intro}

Catalan numbers, in one of their incarnations, count the number of Dyck paths, that is, 
the lattice paths in a square which never cross the diagonal. In recent years, a number of interesting 
results and conjectures \cite{Ar11,ALW14,BGLX14,CDH15,GM13,GM14,GMV14,GN15,Mellit,X15} about ``rational Catalan combinatorics'' have been formulated.
An $(n,m)$-Dyck path is a lattice path in an $n\times m$ rectangle, going from the bottom-right corner $(m,0)$ to the top-left corner $(0,n)$ and never going above the diagonal, which is the line that connects them.
We will denote the set of all $(n,m)$-Dyck paths by $\Ynm$. For coprime $m$ and $n$ there are a number 
of interesting maps involving $\Ynm$, see Figure \ref{commutative triangle}:
\begin{itemize}
\item[(a)] J. Anderson constructed a bijection $\An$ between $\Ynm$ and the set $\Core_{n,m}$ of simultaneous 
$(n,m)$-core partitions.
\item[(b)] Armstrong, Loehr, and Warrington defined a ``sweep'' map $\zeta:\Ynm\to \Ynm$ and conjectured that it is bijective. 
This conjecture was proved by Thomas and Williams in \cite{TW15}.
\item[(c)] The first two authors defined two maps $\CD$ and $\CG$ between $\Ynm$ and the set $\Mnm$ of $\nm$-invariant subsets of $\ZZ$
containing $0$. If combined with a natural bijection between $\Core_{n,m}$ and $\Mnm$, the map $\CD$ coincides with $\An$. Furthermore,
one can prove that $\zeta=\CG\circ \CD^{-1}$. As a consequence, the map $\CG$ is also bijective.
\end{itemize}


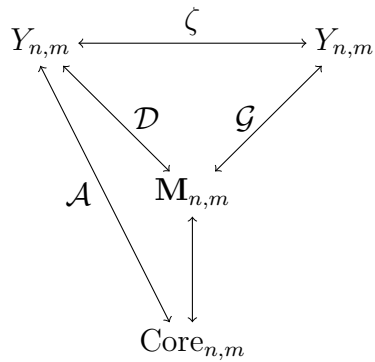
\begin{figure}[ht!]
\begin{center}
\begin{tikzpicture}
\draw (0,0) node {$\Ynm$};
\draw (2,-2) node {$\Mnm$};
\draw (4,0) node {$\Ynm$};
\draw (2,-4) node {$\Core_{n,m}$};
\draw [<->] (1.7,-1.7)--(0.3,-0.3);
\draw [<->] (0.5,0)--(3.5,0);
\draw [<->] (2.3,-1.7)--(3.7,-0.3);
\draw [<->] (2,-2.3)--(2,-3.7);
\draw [<->] (1.7,-3.7)--(0,-0.3);
\draw (1.4,-1) node {$\CD$};
\draw (2.7,-1) node {$\CG$};
\draw (2,0.3) node {$\zeta$};
\draw (0.5,-2) node {$\An$};
\end{tikzpicture}
\end{center}
\caption{Rational Catalan maps in the coprime case
\label{commutative triangle}}
\end{figure}

The goal of the present paper is a partial generalization of 
the diagram in 
Figure \ref{commutative triangle} to the non-coprime case. Let $(n,m)$ be relatively prime, and $d$ be a positive integer. Let $N=dn$ and $M=dm.$  The set $\YNM$
is well defined for all $n,m,d$, and the definition of $\zeta$ can be carried over with minimal changes. However, while 
the sets $\Core_{N,M}$ and $\MNM$ are still in bijection with
each other, the sets become infinite. Indeed, an $\NM$-invariant subset of $\ZZ$ can be identified with a collection of $d$ $\nm$-invariant subsets, one for  each remainder $\bmod\ d$.
These subsets won't necessarily have minimum element $0$, and so
we will want to shift or translate each a fixed amount.
we will want to shift or translate each a fixed amount.
More abstractly, this defines a map $\epsilon:\MNM\to (\Mnm)^{d}$ and different shifts correspond to different preimages under $\epsilon$.

To resolve this problem, we introduce a certain equivalence relation
$\sim$ on $\MNM$.  It satisfies that $\Delta_1\sim \Delta_2$ implies
$\bepsilon(\Delta_1)=\bepsilon(\Delta_2)$, where
$\bepsilon:\MNM\to (\Mnm)^{d} \to (\Mnm)^{d}/\mathord\Sd$,
 so $\bepsilon$ is well
defined on $\MNM/\mathord\sim$.  The following theorem is the main
result of the paper.

\begin{theorem}
For all positive $N,M$ one can define  maps $$\CD,\CG:\MNM/\mathord\sim\longrightarrow \YNM$$ such that the following 
results hold:
\begin{itemize}
\item[(a)] The maps $\CD$ and $\CG$ are bijective.
\item[(b)] The ``sweep'' map factorizes similarly to the coprime case: $\zeta=\CG\circ \CD^{-1}$.
\item[(c)] Let $d=\gcd(N,M),\ n=N/d,$ and $m=M/d.$ The composition 
$$
\col_d:=\CD^{d}\circ\epsilon\circ \CD^{-1}:\YNM\to  (\Ynm)^d/\mathord\Sd,
$$ 
can be described as follows:
color
the $N+M$ steps in an $(N,M)$-Dyck path with $d$ colors,
i.e., by $\Z/d\Z$, so that there are $n+m$ steps of the
same color $i$, and these steps
will 
form an $(n,m)$-Dyck path
after possibly translating connected components by  integer multiples
 of $\overrightarrow{(m,-n)}$ to make the $i$-colored steps  connected.
\end{itemize}
\end{theorem}

As we do not have a canonical way of assigning colors, we
must pass to $\Sd$ orbits above.
We shall see in Section \ref{sec-equivrel} that the coloring is finer
than $\Sd$ orbits and in fact corresponds to an isomorphims class of a
labeled directed  graph with $d$ nodes.


\begin{figure}[ht!]
\begin{center}
\begin{tikzpicture}
\draw (0,0) node {$\YNM$};
\draw (2.1,-2) node {$\MNM/\mathord\sim$};
\draw (4,0) node {$\YNM$};
\draw (2,-4) node {$\left(\Mnm\right)^d/\mathord\Sd$};
\draw (-0.3,-2) node {$\left(\Ynm\right)^d/\mathord\Sd$};
\draw [<->] (1.7,-1.7)--(0.3,-0.3);
\draw [->] (0.5,0)--(3.5,0);
\draw [->] (2.3,-1.7)--(3.7,-0.3);
\draw [->] (2,-2.3)--(2,-3.7);
\draw [<->] (1.7,-3.7)--(0.3,-2.3);
\draw [->] (0,-0.3)--(0,-1.7);
\draw (1.4,-1) node {$\CD$};
\draw (2.7,-1) node {$\CG$};
\draw (2,0.3) node {$\zeta$};
\draw (2.3,-3) node {$\bepsilon$};
\draw (0.8,-3.3) node {$\CD^d$};
\draw (-0.5,-1) node {$\col_d$};
\end{tikzpicture}
\end{center}
\caption{Rational Catalan maps in the non-coprime case.
\label{commutative triangle 2} }
\end{figure}
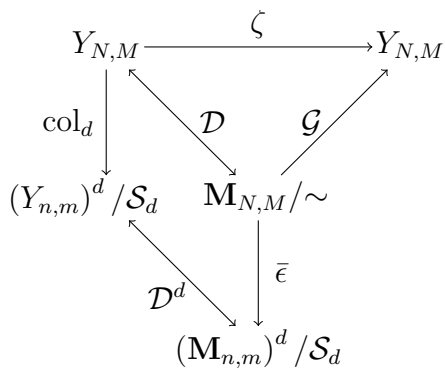

We illustrate all these maps in Figure \ref{commutative triangle 2}. We also give an explicit description of the ``coloring map'' $\col_d$,
as well as its inverse given proper gluing data.
In the ``classical'' case $M=N$ we get $d=N$ and $m=n=1$, therefore $\col_d$ colors a Dyck path in $n$ colors such that the
pairs of steps of the same color form a $(1,1)$-Dyck path. In this case the coloring is equivalent to presenting a Dyck path as a regular sequence of parentheses, with every opening and its corresponding closing parenthesis corresponding to the pair of steps of the same color.

We conjecture a relation between the  constructions of this paper, combinatorial identities and link invariants. 
Recall that the ``compositional rational shuffle conjecture" of \cite{BGLX14} (proved in \cite{Mellit}) relates a certain sum over $(N,M)$-Dyck paths to certain matrix elements of operators acting on symmetric functions. Here we propose a different sum over $(N,M)$-invariant subsets,
 and plan to clarify the relation between the two in the future work. We define the generating series:
\begin{equation}
 C_{N,M}(q,t)=\sum_{\Delta\in \MNM}q^{\gap(\Delta)}t^{\dinv(\Delta)},
\end{equation}
where
$$\gap(\Delta)=|\BZ_{\ge 0}\setminus \Delta|.$$
For $d=1$  it agrees with the rational $q,t$-Catalan polynomial \cite{GM13,Ar11}
$$
c_{N,M}(q,t)=\sum_{D\in \YNM}q^{\area(D)}t^{\dinv(D)},
$$
and it follows from the results of \cite{Mellit}
that:
\begin{equation}
\label{shuffle coprime}
C_{n,m}(q,t)=c_{n,m}(q,t)=\sum_{D\in \YNM}q^{\area(D)}t^{\dinv(D)}=(P_{n,m}(1),h_n).
\end{equation}
Here $P_{n,m}$ is a certain operator defined in \cite{GN15,BGLX14} and acting on the space of symmetric functions. 
In particular, the left hand side of \eqref{shuffle coprime} is symmetric in $q$ and $t$.
It was also proved in \cite{GN15} that the right hand side of \eqref{shuffle coprime} equals the ``refined Chern-Simons invariant" (in the sense of \cite{AS}) of the $(n,m)$ torus knot, and conjectured that it equals the Poincar\'e polynomial of the $(a=0)$ part of the Khovanov-Rozansky homology \cite{KhR} of this knot.

For $d>1$, the formula for $c_{N,M}(q,t)$ generalizing \eqref{shuffle coprime} was conjectured in \cite{BGLX14} and proved in \cite{Mellit}.
However, $C_{N,M}$ is now an infinite power series while $c_{n,m}$ is a finite polynomial.

\begin{conjecture}
\label{conj: links}
For general $d\ge 1$, the following statements hold:
\begin{itemize}
\item[(a)] One has $C_{N,M}(q,t)=\frac{1}{(1-q)^{d-1}}(P_{n,m}^{d}(1),h_{N})$, where $P_{n,m}$ is the same operator as in \eqref{shuffle coprime}.
\item[(b)] The series $C_{N,M}(q,t)/(1-q)$ agrees with the Poincar\'e series of the $(a=0)$ part of the Khovanov-Rozansky homology of the $(N,M)$ torus link.
\end{itemize}
\end{conjecture}

The part (a) immediately implies that $C_{N,M}(q,t)(1-q)^{d-1}$ is symmetric in $q$ and $t$. To support the conjecture, we use a recent result of Elias and Hogancamp \cite{EH} to prove the following:

\begin{theorem}
\label{EH comparison}
Conjecture \ref{conj: links}(b) holds for $M=N$.
\end{theorem} 

In the case $M=N$, part (a) of the conjecture is equivalent to \cite[Conjecture 1.15]{EH} (see also \cite{Wi}), 
but, to our knowledge, it is still open. For general $M$ and $N$, it fits into the framework of conjectures of
\cite{AS,GN15,GNR}, and we refer the reader to these references for more details.

\section*{Acknowledgements}
\label{sec-ack}

We would like to thank Fran\c{c}ois Bergeron and Nathan Williams for the useful discussions. A preliminary version of the paper was reported at the FPSAC 2016 conference \cite{GMV16}. The work of E.\,G. was partially supported by the NSF grant DMS-1559338, Hellman fellowship, grant RSF 16-11-10160 and Russian Academic Excellence Project ’5-100’.
NSF grant DMS-1559338 partially supported collaborative visits
by M.\,M. M.M. participation in FPSAC 2016 was supported by a KSU start-up grant.

\section{Relatively prime case} 
\label{sec-coprime}

Let $(n,m)$ be a pair of relatively prime positive integers. Consider an $n\times m$ rectangle $R_{n,m}.$ Let $\Ynm$ be the set of Young diagrams that fit under the diagonal in $R_{n,m}.$ We will often abuse notation by identifying a diagram $D\in \Ynm$ with its boundary path (sometimes also called a {\it rational Dyck path}), and with the corresponding partition. We will also think about the rectangle $R_{n,m}$ as a set of boxes, identified with a subset in $\ZZ$ with the 
bottom-left
corner box identified with $(0,0).$ In our convention,
 $n$ is the height of $R_{n,m}$ and $m$ is its width;
and
the boundary path of $D\subset R_{n,m}$ follows the boundary from the
bottom-right corner to the top-left corner. 
See
Example \ref{ex-zeta} below.
In Section \ref{sec-glue}, it will also be convenient to
identify the path $D$ with a function (or its plot) $[0, n+m] \to \BR^2$.

There are two important combinatorial statistics on the set $\Ynm:$ $\area$ and $\dinv.$

\begin{definition}
Let $D\in \Ynm.$ Then $\area(D)$ is equal to the number of whole boxes that fit between the diagonal of $R_{n,m}$ and the boundary path of $D.$
\end{definition}

Note that $\area(D)$ ranges from $0$ for the full diagram to
$$\delta=\frac{(m-1)(n-1)}{2}$$
for the empty diagram.
The $\coarea(D) = \delta - \area(D)$ is then just the number
of boxes in the Young diagram $D$. 
 One natural approach to the $\dinv$ statistic is to define the map $\zeta:\Ynm\to \Ynm$ and then set $\dinv(D):=\area(\zeta(D)).$ In the case $m=n+1$ the map $\zeta$ was first defined by Haglund (\cite{Hd08}), then it was generalized by Loehr to the case $m=kn+1$ for any $k\in \ZZ$  (\cite{L05}), and to the general case of any relatively prime $(n,m)$ by Gorsky and Mazin in \cite{GM13}. In \cite{ALW14} it was put into even larger framework of so called sweep maps. Below is one of the equivalent possible definitions.

\begin{definition} The {\it rank} of a box $(x,y)\in \BZ^2$ is given by the linear function
$$
\rank(x,y)=mn-m-n-nx-my.
$$
\end{definition}

Note that the boxes of non-negative ranks are exactly those that fit under the bottom-right to top-left diagonal of $R_{n,m}.$ Let $D\in \Ynm.$ One ranks the steps of the boundary path of $D$ as follows.

\begin{definition}
The rank of a vertical step of $D$ is equal to the rank of the box immediately to the left of it. The rank of a horizontal step is equal to the rank of the box immediately above it.
\end{definition}

\noindent
In other words, the ranks of steps can be defined inductively as follows. We follow the boundary path of $D$ starting from the bottom-right corner. The first step is ranked $-m.$ Otherwise,  the rank of each step equal to the rank of the previous step plus $n,$ if the previous step is horizontal, and it equals to the rank of the previous step minus $m,$ if the previous step is vertical.
Note the last step is ranked $0$ (and is vertical).

Note that for relatively prime $(n,m)$ all the ranks of the steps of a diagram $D\in \Ynm$ are distinct.

\begin{definition}
The boundary path of the diagram $\zeta(D)$ is obtained from the boundary path of $D$ by rearranging the steps in the increasing order of ranks.
\end{definition}
 
The definition of the map $\zeta$ is illustrated in Example \ref{ex-zeta}. One can verify that the diagram $\zeta(D)$ fits under the diagonal of $R_{n,m}$ (see \cite{GM13} and \cite{ALW14}). The following result is considerably harder, see also \cite{GM14, Hd08, L05, X15} for partial results in this direction.

\begin{theorem}[\cite{TW15}]
The map $\zeta$ is bijective.
\end{theorem}
The following approach to studying the map $\zeta$ was suggested
in \cite{GM13}.

\begin{definition}

We say that a subset $\Delta\subset\mathbb Z_{\ge 0}$ is $\nm$-invariant and $0$-normalized if $\Delta+m\subset\Delta,$ $\Delta+n\subset\Delta,$ and $\min(\Delta)=0.$
Let $\Mnm$ be the set of all such subsets $\Delta$.
\end{definition}

In \cite{GM13} two maps $\CD$ and $\CG$ from the set $\Mnm$ to $\Ynm$ were constructed. 

\begin{definition}
Let $\Delta\in \Mnm.$ The diagram $\CD(\Delta)$ consists of all boxes in $R_{n,m}$
whose ranks belong to $\Delta$.
\end{definition}

Clearly, $\CD(\Delta)$ fits under the diagonal.
In particular, one gets that $\CD(\Gamma_{n,m})=\emptyset,$ where $\Gamma_{n,m}:=\{an+bm\ |\ a,b\in\ZZ\}$ is the semigroup generated by $n$ and $m,$ and $\CD(\ZZ)$ is the full diagram containing all the boxes below the diagonal.  Note that the $\nm$-invariance of $\Delta$ implies that $\CD(\Delta)$ is indeed a Young diagram. Note also that $\CD$ is a bijection. Indeed, it is  not hard to see that rank provides a bijection between the boxes below the diagonal in $R_{n,m}$ and the integers in $\ZZ\setminus \Gamma_{n,m}.$ 

It is also important to sometimes consider the {\em periodic extension} $P(\Delta)$ of the boundary path of $\CD(\Delta).$ Equivalently, it can be defined as the infinite lattice path separating the boxes in $\Z^2$ which ranks belong to $\Delta$ from the boxes which ranks belong to the complement $\Z\backslash\Delta.$ We will call such paths $(n,m)$-periodic.
 See Figure \ref{Figure: diagrams coprime} for an example.  

\begin{remark}
J. Anderson in \cite{A02} defined a bijection between $\Ynm$ and the set $\Core_{n,m}$ of $(n,m)$-cores, that is, Young diagrams with no hooks of length $n$ or $m$. The  standard bijection between $\Core_{n,m}$ and $\M_{n,m}$ identifies Anderson's bijection with the map $\CD$, see e.g \cite{GM14} for details.
\end{remark}

\begin{definition}
The numbers $0=a_0<a_1<\ldots<a_{n-1},$ such that 
$$
\{a_0,\ldots,a_{n-1}\}=\Delta\setminus (\Delta+n)
$$
are called {\it the $n$-generators of $\Delta.$} The numbers $\{b_0<b_1<\ldots<b_{m-1}\}$ such that  
$$
\{b_0,\ldots,b_{m-1}\}=(\Delta-m)\setminus \Delta
$$
are called {\it the $m$-cogenerators of $\Delta.$}
\end{definition}

\begin{remark}
Let $D=\CD(\Delta).$ The ranks of the vertical steps of $D$ are exactly the $n$-generators of $\Delta,$ and the ranks of the horizontal steps of $D$ are exactly the $m$-cogenerators of $\Delta.$
We will often mark $n$-generators by $\times$ and $m$-cogenerators
by $\square$.
\end{remark}

\begin{definition}\label{Definition: map G}
The diagram $\CG(\Delta)$ has row lengths $g_0,\ldots,g_{n-1}$ given by the following formula:
$$
g_k=\sharp\{b_i\ |\ b_i>a_k\}.
$$
Equivalently, the boundary path of $\CG(\Delta)$ can be obtained by rearranging the set $$S=\{a_0,\ldots,a_{n-1},b_0,\ldots,b_{m-1}\}$$ in increasing order and replacing $n$-generators by 
vertical 
steps and $m$-cogenerators by 
horizontal 
steps, from bottom right to top left. 
\end{definition}

The next result follows   from the above definitions.
\begin{proposition}\cite{GM13, GM14}
The following identity holds:
$$
\zeta(D)=\CG\circ \CD^{-1}(D).
$$
\end{proposition}
\begin{corollary}
Since $\zeta$ and $\CD$ are bijective, the map $\CG$ is a bijection too.
\end{corollary}

\begin{example} \label{ex-zeta}

For example, if $n=5,$ $m=3,$ and
$\Delta=\{0,3,5,6,7,8,\dots\}$ then the $5$-generators of $\Delta$ are $0,3,6,7,9$
and $3$-cogenerators are $-3,2,4.$
The diagram $\CD(\Delta)$ consists of one box,
which has rank $7.$ The ranked boundary path of $D$ is 
$$
\begin{array}{cccccccc}
         \horc&\horc&\verc&\horc&\verc&\verc&\verc&\verc\\

-3&2&7&4&9&6&3&0
\end{array}
$$
read bottom to top,
which we sort to the boundary path of $\zeta(D)$
$$
\begin{array}{cccccccc}
\horc& \verc&  \horc& \verc&\horc &\verc&\verc&\verc\\
-3&0&2&3&4&6&7&9
\end{array}
$$
See Figure \ref{Figure: diagrams coprime} for the diagrams $D$ and $\zeta(D).$ Note, that if one takes the union of the $5$-generators and $3$-cogenerators and reads them in the increasing order, then one gets $-3,0,2,3,4,6,7,9.$ Replacing generators by ``$\ver$'' and cogenerators by ``$\hor$'', one gets $\mathtt{hvhvhvvv},$ which is the boundary path of $\zeta(D).$

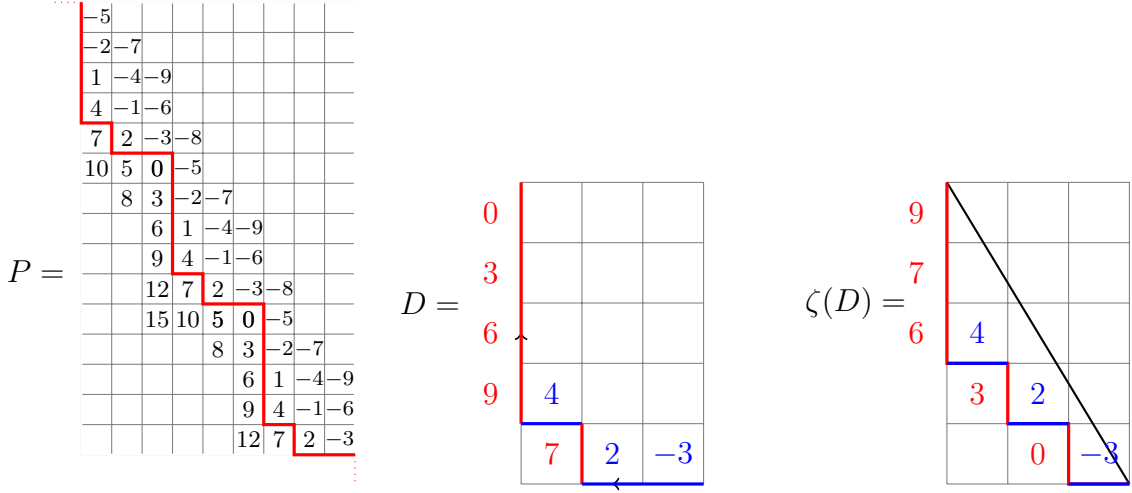
\begin{figure}
\begin{center}
\begin{tikzpicture}[scale=0.4]
\draw [step=1.0, gray] (-3,-5) grid (6,10);
\draw [white, thick] (-3,-5)--(4,-5); 
\draw [white, thick] (-3,10)--(6,10);
\draw [white, thick] (-3,-5)--(-3,6); 
\draw [white, thick] (6,-5)--(6,10);

\draw (-0.5,4.5) node {\scriptsize{$0$}};
\draw (-0.5,3.5) node {\scriptsize{$3$}};
\draw (-0.5,2.5) node {\scriptsize{$6$}};
\draw (-0.5,1.5) node {\scriptsize{$9$}};
\draw (-0.5,0.5) node {\scriptsize{$12$}};
\draw (0.5,5.5) node {\tiny{$-8$}};
\draw (0.5,4.5) node {\tiny{$-5$}};
\draw (0.5,3.5) node {\tiny{$-2$}};
\draw (0.5,2.5) node {\scriptsize{$1$}};
\draw (0.5,1.5) node {\scriptsize{$4$}};
\draw (0.5,0.5) node {\scriptsize{$7$}};
\draw (1.5,3.5) node {\tiny{$-7$}};
\draw (1.5,2.5) node {\tiny{$-4$}};
\draw (1.5,1.5) node {\tiny{$-1$}};
\draw (1.5,0.5) node {\scriptsize{$2$}};
\draw (2.5,2.5) node {\tiny{$-9$}};
\draw (2.5,1.5) node {\tiny{$-6$}};
\draw (2.5,0.5) node {\tiny{$-3$}};
\draw (-.5,-0.5) node {\scriptsize{$15$}};
\draw (.5,-0.5) node {\scriptsize{$10$}};
\draw (1.5,-0.5) node {\scriptsize{$5$}};
\draw (2.5,-0.5) node {\scriptsize{$0$}};
\draw (1.5,-0.5) node {\scriptsize{$5$}};
\draw (1.5,-1.5) node {\scriptsize{$8$}};

\draw (-2.5,9.5) node {\tiny{$-5$}};
\draw (-2.5,8.5) node {\tiny{$-2$}};
\draw (-2.5,7.5) node {\scriptsize{$1$}};
\draw (-2.5,6.5) node {\scriptsize{$4$}};
\draw (-2.5,5.5) node {\scriptsize{$7$}};
\draw (-1.5,8.5) node {\tiny{$-7$}};
\draw (-1.5,7.5) node {\tiny{$-4$}};
\draw (-1.5,6.5) node {\tiny{$-1$}};
\draw (-1.5,5.5) node {\scriptsize{$2$}};
\draw (-0.5,7.5) node {\tiny{$-9$}};
\draw (-0.5,6.5) node {\tiny{$-6$}};
\draw (-0.5,5.5) node {\tiny{$-3$}};
\draw (-2.5,4.5) node {\scriptsize{$10$}};
\draw (-1.5,4.5) node {\scriptsize{$5$}};
\draw (-0.5,4.5) node {\scriptsize{$0$}};
\draw (-1.5,3.5) node {\scriptsize{$8$}};

\draw (2.5,-0.5) node {\scriptsize{$0$}};
\draw (2.5,-1.5) node {\scriptsize{$3$}};
\draw (2.5,-2.5) node {\scriptsize{$6$}};
\draw (2.5,-3.5) node {\scriptsize{$9$}};
\draw (2.5,-4.5) node {\scriptsize{$12$}};
\draw (3.5,0.5) node {\tiny{$-8$}};
\draw (3.5,-0.5) node {\tiny{$-5$}};
\draw (3.5,-1.5) node {\tiny{$-2$}};
\draw (3.5,-2.5) node {\scriptsize{$1$}};
\draw (3.5,-3.5) node {\scriptsize{$4$}};
\draw (3.5,-4.5) node {\scriptsize{$7$}};
\draw (4.5,-1.5) node {\tiny{$-7$}};
\draw (4.5,-2.5) node {\tiny{$-4$}};
\draw (4.5,-3.5) node {\tiny{$-1$}};
\draw (4.5,-4.5) node {\scriptsize{$2$}};
\draw (5.5,-2.5) node {\tiny{$-9$}};
\draw (5.5,-3.5) node {\tiny{$-6$}};
\draw (5.5,-4.5) node {\tiny{$-3$}};

\draw [red, very thick] (-3,10)--(-3,6)--(-2,6)--(-2,5)--(0,5)--(0,1)--(1,1)--(1,0)--(3,0)--(3,-4)--(4,-4)--(4,-5)--(6,-5);
\draw (-4.5,1.1) node {$P=$};
\draw [red, dotted] (-3,10)--(-4,10); 
\draw [red, dotted] (6,-5)--(6,-6); 
\end{tikzpicture}
\ \ 
\begin{tikzpicture}[scale=0.8]
\draw [step=1.0, gray] (0,0) grid (3,5);

\draw [black, thick, ->] (0,2)--(0,2.5);
\draw [black, thick, ->] (2,0)--(1.5,0);
\draw [red] (-0.5,4.5) node {$0$};
\draw  [red]  (-0.5,3.5) node {$3$};
\draw   [red] (-0.5,2.5) node {$6$};
\draw  [red]  (-0.5,1.5) node {$9$};
\draw  [blue] (0.5,1.5) node {$4$};
\draw  [red] (0.5,0.5) node {$7$};
\draw  [blue] (1.5,0.5) node {$2$};
\draw  [blue] (2.5,0.5) node {$-3$};

\draw [red, very thick] (0,5)--(0,1);
\draw [blue, very thick] (0,1)--(1,1);
\draw [red, very thick] (1,1)--(1,0);
\draw [blue, very thick] (1,0)--(3,0);
\draw (-1.5,3) node {$D=$};
\draw [step=1.0, gray] (7,0) grid (10,5);

\draw [black, thick] (7,5)--(10,0);

\draw [red, very thick] (7,5)--(7,2)--(8,2)--(8,1)--(9,1)--(9,0)--(10,0);
\draw (5.5,3) node {$\zeta(D)=$};
\draw [blue, very thick] (7,2)--(8,2);
\draw [blue, very thick] (8,1)--(9,1);
\draw [blue, very thick] (9,0)--(10,0);
\draw [red] (8.5,0.5) node {$0$};
\draw  [red]  (7.5,1.5) node {$3$};
\draw   [red] (6.5,2.5) node {$6$};
\draw  [red] (6.5,3.5) node {$7$};
\draw  [red]  (6.5,4.5) node {$9$};
\draw  [blue] (8.5,1.5) node {$2$};
\draw  [blue] (9.5,0.5) node {$-3$};
\draw  [blue] (7.5,2.5) node {$4$};
\end{tikzpicture}
\end{center}
\caption{Here $n=5$ and $m=3.$ On the left is (a fragment of) the periodic path $P=P(\{0,3,5,6,7,8,\dots\}),$ center is the diagram $D=\CD(\{0,3,5,6,7,8,\dots\}),$ and on the right is the diagram $\zeta(D).$\label{Figure: diagrams coprime}}
\end{figure}

\end{example}

The approach with invariant subsets allows one to relate the $\dinv$ statistic to geometry. Let $V=\mathbb C[t]/t^{2\delta}\mathbb C[t]$ be the ring of polynomials of degree less than $2\delta=(m-1)(n-1).$ Let $Gr(\delta,V)$ be the Grassmannian of half-dimensional subspaces in $V.$ Consider the subvariety $J_{n,m}\subset Gr(\delta,V)$ consisting of subspaces in $V$ invariant under multiplication by $t^m$ and $t^n:$
$$
J_{n,m}=\{U\in Gr(\delta,V):t^mU\subset U,\ t^nU\subset U\}.
$$ 
These varieties appear in algebraic geometry as local versions of the compactified Jacobians (see Beauville \cite{Be99} and Piontkowski \cite{P07}), and in representation theory as homogeneous affine Springer fibers, where they were first considered by Lusztig and Smelt in \cite{LS91} and then by Piontkowski \cite{P07}. Both Lusztig and Smelt, and Piontkowski showed that $J_{n,m}$ has a natural decomposition into complex affine cells enumerated by elements
of $\Mnm.$ Moreover, the dimension of the cell $C_{\Delta}$ corresponding to an invariant subset $\Delta\in \Mnm$ is given by 
$$\dim_{\mathbb C} C_{\Delta}=|\CG(\Delta)|=\delta-\dinv(\Delta).$$

Therefore, one gets the following theorem.

\begin{theorem}[\cite{GM13}]
The Poincar\'e polynomial $P_{n,m}(t)$ of the variety $J_{n,m}$ is given by
$$
P_{n,m}(t)=\sum\limits_{D\in \Ynm}t^{2(\delta-\dinv(D))}.
$$ 
\end{theorem}

Moreover, bijectivity of the map $\zeta$ (or, equivalently, the map $\CG$) implies a simpler formula:

$$
P_{n,m}(t)=\sum\limits_{D\in \Ynm}t^{2|D|},
$$
where $|D|=\delta-\area(D)$ is simply the number of boxes in $D.$

\section{Non-relatively prime case.}
\label{sec-noncoprime}

\subsection{Sweep map}
\label{sec-sweep}

The notion of a rational Dyck path naturally generalizes to the non relatively prime case. Let $(n,m)$ be relatively prime, and $d$ be a positive integer. Let $N=dn$ and $M=dn.$ Consider an $N\times M$ rectangle $R_{N,M}$ and the set $\YNM$ of Young diagrams that fit under the diagonal in $R_{N,M}.$ The $\area$ statistic can be generalized directly. The $\dinv$ statistic and the map $\zeta$ are a bit more tricky. It is convenient to adjust the rank function on the boxes in the following way:
$$
\rank(x,y)=dmn-m-n-nx-my.
$$
The steps of the boundary path of a diagram $D\in \YNM$ are ranked as before with respect to the new rank function.  The first step is still 
ranked $-m$ and the inductive description of the ranks still holds with respect to $+n,-m$; it still holds that the boxes with non-negative rank
are those below the diagonal.
However, for $d>1$ some distinct steps might have the same  rank, therefore rearranging the steps of the path according to their rank is problematic. The following idea for overcoming this difficulty was suggested by Fran\c{c}ois Bergeron. It can also be found in \cite{ALW14}.

\begin{definition}
Let $D\in \YNM.$ The boundary path of the diagram $\zeta(D)$ is obtained from the boundary path of $D$ by rearranging the steps so their ranks
are weakly increasing. If two steps have the same rank, then they are ordered in the reversed order of appearance in the boundary path of $D.$
\end{definition}

\begin{example}\label{Example: zeta nonrp}
Consider the diagram $D\in Y_{9,6}$ with the boundary path
$\mathtt{hvhvvhhhvhvvvvv}$
(see Figure \ref{Figure: diagrams non-coprime}). The ranked boundary path of $D$ is 
$$
\begin{array}{ccccccccccccccc}
\hor&
\bver&
\hor&
\ver&
                \ver&
\hor&
\bhor&
\bhor&
\ver&
         \hor& 
\ver&
\ver&
\bver&
\ver&
\ver
\\
-2 & 1 & -1 & 2 & 0 & -2 & 1 & 4 & 7 & 5 & 8 & 6 & 4 & 2 & 0
\end{array}
$$
which we sort to the boundary path of $\zeta(D)$
$$
\begin{array}{ccccccccccccccc}
\hor&
\hor&
\hor&
                \ver&
\ver&
\bhor&
\bver&
\ver&
\ver&
         \bver&
\bhor&
\hor&
\ver&
\ver&
\ver
\\
-2
&
-2
&
-1
&
0
&
0
&
1
&
1
&
2
&
2
&
4
&
4
&
5
&
6
&
7
&
8.
\end{array}
$$
Note there are two steps of rank $4$ in the boundary path of $D:$
when read bottom to top on the path but left to right above,
first there is a horizontal step, and then there is a
vertical step. In the boundary path of  $\zeta(D)$ the order of these two steps is reversed. Similarly for the two steps of rank $1.$


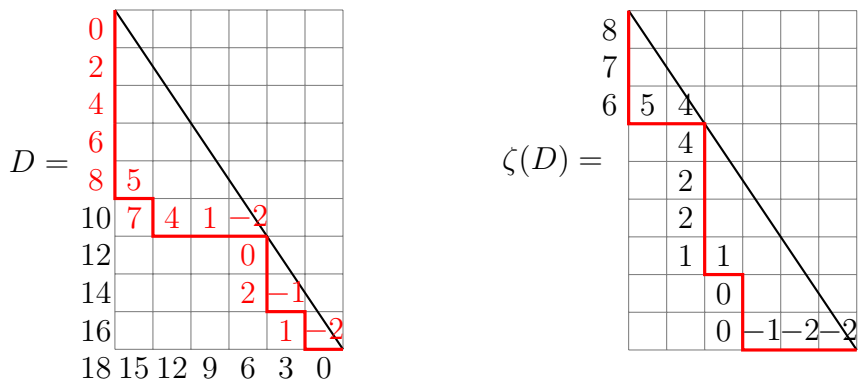
\begin{figure}
\begin{center}
\begin{tikzpicture}[scale=0.5]
\draw [step=1.0, gray] (0,0) grid (6,9);

\draw [black, thick] (0,9)--(6,0);

\draw [red, very thick] (0,9)--(0,6);
\draw [red, very thick] (2,3)--(4,3);
\draw [red, very thick] (0,6)--(0,4)--(1,4)--(1,3)--(2,3);
\draw [red, very thick] (4,3)--(4,1)--(5,1)--(5,0)--(6,0);

\draw [red] (-0.5,8.5) node {$0$};
\draw [red] (-0.5,7.5) node {$2$};
\draw [red] (-0.5,6.5) node {$4$};
\draw [red] (-0.5,5.5) node {$6$};
\draw [red] (-0.5,4.5) node {$8$};
\draw [red] (0.5,4.5) node {$5$};
\draw [red] (0.5,3.5) node {$7$};
\draw [red] (1.5,3.5) node {$4$};
\draw [red] (2.5,3.5) node {$1$};
\draw [red] (3.5,3.5) node {$-2$};
\draw [red] (3.5,2.5) node {$0$};
\draw [red] (3.5,1.5) node {$2$};
\draw [red] (4.5,1.5) node {$-1$};
\draw [red] (4.5,0.5) node {$1$};
\draw [red] (5.5,0.5) node {$-2$};
\draw (-0.5,0.5) node {$16$};
\draw (-0.5,-0.5) node {$18$};
\draw (-0.5,1.5) node {$14$};
\draw (-0.5,2.5) node {$12$};
\draw (-0.5,3.5) node {$10$};
\draw (0.5,-0.5) node {$15$};
\draw (1.5,-0.5) node {$12$};
\draw (2.5,-0.5) node {$9$};
\draw (3.5,-0.5) node {$6$};
\draw (4.5,-0.5) node {$3$};
\draw (5.5,-0.5) node {$0$};
\draw (-2,5) node {$D=$};

\end{tikzpicture}
\qquad \qquad
\begin{tikzpicture}[scale=0.5]
\draw [step=1.0, gray] (0,0) grid (6,9);

\draw [black, thick] (0,9)--(6,0);

\draw [red, very thick] (0,9)--(0,6)--(2,6)--(2,2)--(3,2)--(3,0)--(6,0);

\draw  (1,-0.75) node {};

\draw  (-0.5,8.5) node {$8$};
\draw  (-0.5,7.5) node {$7$};
\draw  (-0.5,6.5) node {$6$};
\draw  (0.5,6.5) node {$5$};
\draw  (1.5,6.5) node {$4$};
\draw  (1.5,5.5) node {$4$};
\draw  (1.5,4.5) node {$2$};
\draw (1.5,3.5) node {$2$};
\draw (1.5,2.5) node {$1$};
\draw (2.5,2.5) node {$1$};
\draw (2.5,1.5) node {$0$};
\draw (2.5,0.5) node {$0$};
\draw (3.5,0.5) node {$-1$};
\draw (4.5,0.5) node {$-2$};
\draw (5.5,0.5) node {$-2$};
\draw (-2,5) node {$\zeta(D)=$};

\end{tikzpicture}
\end{center}
\caption{Here $n=9$ and $m=6.$ On the left is the diagram $D$ with the boundary path 
$\mathtt{hvhvvhhhvhvvvvv}$
marked with ranks;
and on the right is the diagram $\zeta(D).$\label{Figure: diagrams non-coprime}}
\end{figure}

\end{example}

Now the statistic $\dinv$ can be defined as
$$
\dinv(D):=\area(\zeta(D)).
$$
Note that in \cite{BGLX14} a different definition of $\dinv$ for the non relatively prime case is used:
$$
\dinv'(D):=\sharp\left\{\square\in D\colon\frac{\leg(\square)}{\arm(\square)+1}<\frac{n}{m}\le\frac{\leg(\square)+1}{\arm(\square)}\right\}.
$$

\begin{lemma}
\label{lemma-dinv}
One has
$$
\dinv(D)=\dinv'(D)
$$
for any $D\in \YNM.$
\end{lemma}

\begin{proof}
This result essentially follows from Corollary $1$ on page $8$ in \cite{Mn13}.
 For every box $\square\in R_{N,M}$ there is exactly one horizontal step
$\hor_\square$ of the Dyck path $D$ in the same column, and exactly one
vertical step $\ver_\square$ of $D$ in the same row. This provides a bijection
between the boxes in $R_{N,M}$ and the couples: one vertical step of $D$ and one
horizontal step of $D$. The reordering of the steps according to $\zeta$ gives rise to a bijective map $\phi:R_{N,M}\to R_{N,M},$ where the box $\phi(\square)$ corresponds to the pair of steps of $\zeta(D)$ obtained from $\hor_\square$ and $\ver_\square$ by reordering according to $\zeta$.

With the terminology above, $\arm(\square)$ is the number of boxes strictly between
the box $\square\in R_{N,M}$ and the horizontal step $\hor_\square$ of the boundary of $D,$ whereas its $\leg(\square)$ is the number of boxes strictly between the box $\square$ and $\ver_\square$. Observe, $\ver_\square$ appears in the path before
$\hor_\square$ if and only if $\square\in D.$
Consider two cases:
\begin{enumerate}
\item Suppose $\square\in D,$ then one has
$$
\rank(\hor_\square)=\rank(\square)-m(\leg(\square)+1),
$$ 
and
$$
\rank(\ver_\square)=\rank(\square)-n\arm(\square).
$$ 
One gets $\phi(\square)\in\zeta(D)$ if and only if after the reordering the step in $\zeta(D)$ corresponding to $\ver_\square$ comes before the step corresponding to $\hor_\square.$ According to the definition of $\zeta,$ in this case it is equivalent to $\rank(\ver_\square)<\rank(\hor_\square),$ which is in turn equivalent to

$$
\frac{\leg(\square)+1}{\arm(\square)}<\frac{n}{m}.
$$
\item Suppose $\square\in R_{N,M}\backslash D.$ Similarly, one gets
$$
\rank(h_\square)=\rank(\square)+m\leg(\square),
$$ 
and
$$
\rank(v_\square)=\rank(\square)+n(\arm(\square)+1).
$$ 
In this case,
$$
\frac{\leg(\square)}{\arm(\square)+1}\ge \frac{n}{m}
$$
if and only if $\rank(v_\square)\le \rank(h_\square),$ if and only if $\phi(\square)\in\zeta(D).$

\end{enumerate}

Since by definition $\dinv(D)=\sharp R_{N,M}^+-\sharp\zeta(D),$ where $R_{N,M}^+$ is the set of boxes in $R_{N,M}$ that fit under the diagonal, one gets
$$
\dinv(D)=\sharp R_{N,M}^+ -\sharp\left\{\square\in D:\frac{\leg(\square)+1}{\arm(\square)}<\frac{n}{m}\right\}-\sharp\left\{\square\in R_{N,M}\backslash D:\frac{\leg(\square)}{\arm(\square)+1}\ge\frac{n}{m}\right\}
$$
by the above considerations. Corollary $1$ on page $8$ in \cite{Mn13} proves
$$
\sharp\left\{\square\in R_{N,M}\backslash D:\frac{\leg(\square)}{\arm(\square)+1}\ge \frac{n}{m}\right\}=\sharp\left\{\square\in D:\frac{\leg(\square)}{\arm(\square)+1}\ge \frac{n}{m}\right\}+\sharp (R_{N,M}^+\backslash D).
$$ 
Therefore, we conclude that
$$
\dinv(D)=\sharp D-\sharp\left\{\square\in D:\frac{\leg(\square)+1}{\arm(\square)}<\frac{n}{m}\right\}-\sharp\left\{\square\in D:\frac{\leg(\square)}{\arm(\square)+1}\ge \frac{n}{m}\right\}
$$
$$
=\sharp\left\{\square\in D\colon\frac{\leg(\square)}{\arm(\square)+1}<\frac{n}{m}\le\frac{\leg(\square)+1}{\arm(\square)}\right\}=\dinv'(D).
$$
\end{proof}

The cardinality of the sets $\YNM$ of Dyck paths get more complicated in the non-relatively prime. In \cite{Bizley} Bizley shows that
$$
\exp(\sum_{d \ge 1} \frac{1}{d(m+n)} \binom{d(m+n)}{dm} x^d)
$$
is the generating function whose coefficients give the cardinalities of $\YNM$, where $(N,M) = (dn,dm)$ for $\gcd(n,m)=1$. 

On the other hand, the set $\MNM$ of subsets $0\in\Delta\subset\ZZ$ invariant under addition of $M$ and $N$ is infinite when $\gcd(N,M)=d>1.$ Therefore, there is no hope to construct a bijection between the set of such subsets and $\YNM.$
However, the map $\CG: \MNM\to \YNM$ is still well defined. We define an equivalence relation $\sim$ on the set $\MNM,$ so that the relative order of $N$-generators and $M$-cogenerators, and hence the value of $\CG$, is the same within each equivalence class. We then construct a bijection $\CD$ between the equivalent classes $\MNM/\mathord\sim$ and $\YNM,$ so that one gets $\zeta=\CG\circ \CD^{-1}$ as in the $d=1$ case.

Comparing the definition of $\CG$ to Lemma \ref{lemma-dinv}, one can
see that $\dinv$ is constant on the fibres of $\CG$. Further, each
fiber of $\CG$
is a union of
$\sim$ equivalence classes. In fact, each
equivalence class is precisely one fiber by the result of Thomas and Williams\cite{TW15} showing that $\zeta$ is always bijective.

Given an $(N,M)$-invariant subset $\Delta\in \MNM$ one can extract $d$ many $(n,m)$-invariant subsets from it by the following procedure: for each $r\in\{0,1\dots,d-1\}$ consider the subset in $\Delta$ consisting of all integers congruent to $r$ modulo $d,$ subtract $r$ from all these elements and then divide by $d.$ In other words one has
\begin{gather} \label{colorDelta}
\Delta_r=\left[\left(\Delta\cap (d\mathbb Z+r)\right)-r\right]/d.
\end{gather}
Note that the subsets $\Delta_r$ for $r>0$ might not be $0$-normalized. Note also that $\Delta$ can be uniquely reconstructed from $\Delta_0,\dots,\Delta_{d-1},$ so we have a bijection between the set of $0$-normalized $(N,M)$-invariant subsets and (ordered or $\Z/d\Z$-colored) collections of $d$
many $(n,m)$-invariant subsets, such that $\Delta_0$ is zero normalized and $\Delta_i\subset \mathbb Z_{\ge 0}$ for all $i.$

\begin{remark}
\label{rem-core}
There is a natural bijection, extending Anderson's construction,
 between the set of $(N,M)$-invariant subsets
and the set of $(N,M)$-cores. If $\lambda$ is an $(N,M)$-core
corresponding to $\Delta$ then one can check that the $d$-quotient of
$\lambda$ (\cite{Littlewood}) consists of $d$ diagrams each of which are
$(n,m)$-cores.  They naturally correspond to $\Delta_0,\ldots,
\Delta_{d-1}$.
\end{remark}

\subsection{Equivalence relation}
\label{sec-equivrel}

The idea of the equivalence relation is that one should fix the collection $\Delta_0,\dots,\Delta_{d-1}$ up to shifts, but allow them to slide with respect to each other as long as the $N$-generators and $M$-cogenerators of $\Delta$ do not ``jump'' over each other. 
It is motivated by making the invariant sets in the same fiber of $\CG$ equivalent.
Recall the map $\CG$ only cares about the relative order of the
$N$-generators and $M$-cogenerators. We will analyze
an equivalence class by understanding all the positions the generators
and cogenerators can fill while retaining this relative order.
 This analysis will allow us to construct a representative in the equivalence class
of $\Delta \in \MNM$ which has the minimal number of gaps,
and it is on that representative that we can define $\CD$. Later, we will describe the equivalence class of $\Delta$
in terms of rank data from the $\Delta_r$ along with appropriate
gluing data. 
 
Let us  first explain the equivalence relation with an example.

\begin{example}
\label{ex-equiv}
Let $(N,M)=(6,4).$ The following two elements of $\YNM$ are equivalent. Let $\Delta^1$ be given by: 
\begin{center}
\begin{tabular}{cccccccccccccccc}
$0$&$1$&$2$&$3$&$4$&$5$&$6$&$7$&$8$&$9$&$10$&$11$&$12$&$13$&$14$&$\dots$ \\
$\times$&$\circ$&$\square$&$\circ$&$\times$&$\square$&$\bullet$&$\circ$&$\times$&$\times$&$\bullet$&$\square$&$\bullet$&$\times$&$\bullet$&
$\dots$
\end{tabular}
\end{center}
and $\Delta^2:$
\begin{center}
\begin{tabular}{cccccccccccccccc}
$0$&$1$&$2$&$3$&$4$&$5$&$6$&$7$&$8$&$9$&$10$&$11$&$12$&$13$&$14$&$\dots$ \\
$\times$&$\circ$&$\square$&$\circ$&$\times$&$\circ$&$\bullet$&$\square$&$\times$&$\circ$&$\bullet$&$\times$&$\bullet$&$\square$&$\bullet$&$\dots$
\end{tabular}.
\end{center}

\noindent
Here $\times$'s are $6$-generators, $\square$'s are the $4$-cogenerators, $\bullet$'s are other elements of the subset, and $\circ$'s are the other elements of the complement. Note that not all 6 generators and 4 cogenerators fit in the pictures. It is more illustrative
to split $\Delta^1$ into 
its even and odd parts:

\begin{center}
\begin{tabular}{lcccccccccccccccccccccccccc}
$r=0$
&
$0$&&$2$&&$4$&&$6$&&$8$&&$10$&&$12$&&$14$
&$\dots$ \\
&
{\color{red}$\times$}&&{\color{red}$\square$}&&{\color{red}$\times$}&&{\color{red}$\bullet$}&&{\color{red}$\times$}&&{\color{red}$\bullet$}&&{\color{red}$\bullet$}
&&{\color{red}$\dots$}\\
$r=1$ &&
$1$&&$3$&&$5$&&$7$&&$9$&&$11$&&$13$&
&$\dots$ \\
&&
{\color{blue}$\circ$}&&{\color{blue}$\circ$}&&{\color{blue}$\square$}&&{\color{blue}$\circ$}&&{\color{blue}$\times$}&&{\color{blue}$\square$}&
&{\color{blue}$\dots$}\\
\end{tabular}.
\end{center}

\noindent
It is more compact 
to then stack them as

\begin{center}
\begin{tabular}{lccccccccccccc}
$r=0$ &$-4$&$-2$&$0$&$2$&$4$&$6$&$8$&$10$&$12$&$14$&$16$&$18$&$\dots$ \\
&{\color{red}$\square$}&{\color{red}$\circ$}&{\color{red}$\times$}&{\color{red}$\square$}&{\color{red}$\times$}&{\color{red}$\bullet$}&{\color{red}$\times$}&{\color{red}$\bullet$}&{\color{red}$\bullet$}&{\color{red}$\bullet$}&{\color{red}$\bullet$}&{\color{red}$\dots$}\\
$r=1$ &$-3$&$-1$&$1$&$3$&$5$&$7$&$9$&$11$&$13$&$15$&$17$&$19$&$\dots$ \\
&{\color{blue}$\circ$}&{\color{blue}$\circ$}&{\color{blue}$\circ$}&{\color{blue}$\circ$}&{\color{blue}$\square$}&{\color{blue}$\circ$}&{\color{blue}$\times$}&{\color{blue}$\square$}&{\color{blue}$\times$}&{\color{blue}$\bullet$}&{\color{blue}$\times$}&{\color{blue}$\dots$}\\
\end{tabular},
\end{center}
reminiscent of a $d$-abacus.

Finally we just record 
$\Delta^1_0$ and $\Delta^1_1$:

\begin{center}
\begin{tabular}{ccccccccccccccc}
&$-2$&$-1$&$0$&$1$&$2$&$3$&$4$&$5$&$6$&$7$&$8$&$9$&10&$\dots$ \\
{\color{red}$\Delta^1_0$}&{\color{red}$\square$}&{\color{red}$\circ$}&{\color{red}$\times$}&{\color{red}$\square$}&{\color{red}$\times$}&{\color{red}$\bullet$}&{\color{red}$\times$}&{\color{red}$\bullet$}&{\color{red}$\bullet$}&{\color{red}$\bullet$}&{\color{red}$\bullet$}&{\color{red}$\dots$}\\
{\color{blue}$\Delta^1_1$}&{\color{blue}$\circ$}&{\color{blue}$\circ$}&{\color{blue}$\circ$}&{\color{blue}$\circ$}&{\color{blue}$\square$}&{\color{blue}$\circ$}&{\color{blue}$\times$}&{\color{blue}$\square$}&{\color{blue}$\times$}&{\color{blue}$\bullet$}&{\color{blue}$\times$}&{\color{blue}$\dots$}\\
\end{tabular}.
\end{center}

To restore $\Delta^1$ one should multiply both $\Delta^1_0$ and $\Delta^1_1$ by two, add one to $\Delta^1_1,$ and merge them together.
In other words, $\Delta^1 = 2 \Delta^1_0 \cup (1+ 2 \Delta^1_1)$. Similarly, for $\Delta^2$ one gets

\begin{center}
\begin{tabular}{ccccccccccccccc}
&$-2$&$-1$&$0$&$1$&$2$&$3$&$4$&$5$&$6$&$7$&$8$&$9$&$10$&$\dots$ \\
{\color{red}$\Delta^2_0$}&{\color{red}$\square$}&{\color{red}$\circ$}&{\color{red}$\times$}&{\color{red}$\square$}&{\color{red}$\times$}&{\color{red}$\bullet$}&{\color{red}$\times$}&{\color{red}$\bullet$}&{\color{red}$\bullet$}&{\color{red}$\bullet$}&{\color{red}$\bullet$}&{\color{red}$\bullet$}&{\color{red}$\dots$}\\
{\color{blue}$\Delta^2_1$}&{\color{blue}$\circ$}&{\color{blue}$\circ$}&{\color{blue}$\circ$}&{\color{blue}$\circ$}&{\color{blue}$\circ$}&{\color{blue}$\square$}&{\color{blue}$\circ$}&{\color{blue}$\times$}&{\color{blue}$\square$}&{\color{blue}$\times$}&{\color{blue}$\bullet$}&{\color{blue}$\times$}&{\color{blue}$\dots$}\\
\end{tabular}.
\end{center}

Note that the sequences of $N$-generators and $M$-cogenerators are the same for $\Delta^1$ and $\Delta^2,$ even if we take into account the remainder modulo $2.$ In both cases one gets

\begin{gather} \label{ex-equiv-skeleton}
\begin{tabular}{cccccccccc}
{\color{red}$\square$}&{\color{red}$\times$}&{\color{red}$\square$}&{\color{red}$\times$}&{\color{blue}$\square$}&{\color{red}$\times$}&{\color{blue}$\times$}&{\color{blue}$\square$}&{\color{blue}$\times$}&{\color{blue}$\times$}
\end{tabular}
\end{gather}

\noindent
where {\color{red} red} is for even generators and cogenerators
($r=0$), and {\color{blue} blue} is for odd ($r=1$). This is the reason
$\Delta^1 \sim \Delta^2$.
If we only knew the even and odd parts, then,
 in this example, the odd part can be shifted by $1$ with respect to the even part without changing the sequence or parity of
generators and cogenerators. Note that one cannot shift further: in $\Delta^1$ one cannot shift the odd part to the left, and in $\Delta^2$ one cannot shift the odd part to the right and still yield an
invariant set equivalent to $\Delta^1$. 
Also note that while {\color{red}$\Delta^1_0$} $ =$
{\color{red}$\Delta^2_0$},
{\color{blue}$\Delta^1_1$} $=  -1+ ${\color{blue}$\Delta^2_1$}.
  
\end{example}

Let us give a formal definition of the equivalence classes. 

\begin{definition}
\label{def:skeleton}
The {\it skeleton} of an $(N,M)$-invariant subset $\Delta$ is the set 
consisting of its $N$-generators and $M$-cogenerators.
\end{definition}

\begin{example}
The skeleton of  $\Delta^1$ from Example \ref{ex-equiv}
above is $\{-4,0,2,4,5,8,9,11,13,17\}.$
Note is has $10 = 6+4$ elements.

\end{example}

Note that one can uniquely reconstruct an invariant subset $\Delta$ from its skeleton. Indeed, the skeleton contains all the $N$-generators of $\Delta,$ and to distinguish the $N$-generators from the $M$-cogenerators one should simply choose the biggest elements in each congruence class $\bmod\ N.$

An attentive reader may have noticed that the above definition of the
skeleton are not obviously symmetric in $N$ and $M$. It fact, it is
(almost) symmetric by the following lemma.

\begin{lemma}
\label{lem: skeleton}
Let $\Delta$ be some $(N,M)$-invariant subset. An integer $x$ is either an $N$-generator or an $M$-cogenerator of $\Delta$ if
and only if $x+M$ is an $(N+M)$-generator of $\Delta$. 
\end{lemma}

\begin{proof}
Indeed, $x+M$ is an $(N+M)$-generator of $\Delta$ if and only if $x+M\in \Delta$ and $x-N\notin \Delta$.

If $x$ is an $N$-generator then $x\in \Delta$, so $x+M\in \Delta$, but $x-N\notin \Delta$. Hence by the above it is an $(N+M)$-generator.
Assume that $x$ is an $M$-cogenerator. Then $x\notin \Delta$ but $x+M\in \Delta$. If $x-N\in \Delta$ then $x\in \Delta$, contradiction, therefore $x-N\notin \Delta$, and again $x+M$ is an $(N+M)$-generator.

Conversely, assume that $x-N = x+M - (N+M) \notin \Delta$ and $x+M\in \Delta$. If $x\in \Delta$ then $x$ is an $N$-generator,  and if $x\notin \Delta$ then $x$ is a $M$-cogenerator.
\end{proof}

\begin{remark}
One can also prove Lemma \ref{lem: skeleton} using generating functions. Let $f(t)=\sum_{s\in \Delta}t^s$ be the generating function for $\Delta$, then the generating function for the set of $N$-generators equals $(1-t^N)f(t)$ while the generating function for the set of $M$-generators equals $(t^{-M}-1)f(t)$. Therefore the generating function for the skeleton equals:
$$
(1-t^N)f(t)+(t^{-M}-1)f(t)=(t^{-M}-t^N)f(t)=t^{-M}(1-t^{M+N})f(t).
$$
\end{remark}

\begin{corollary}
Let $\Delta \in \MNM$. Then
$x$ is in the $(N,M)$-skeleton of $\Delta$ if and only if $x-N+M$ is in the $(M,N)$--skeleton of $\Delta$.
\end{corollary}

\begin{remark}
Indeed, the distribution of generators and cogenerators in the $(N,M)$ and $(M,N)$ skeletons is different
(say, there are $N$ generators in the former and $M$ in the latter).
One can be obtained from the other by reading the distribution
backwards and swapping $\times$s and $\square$s.
In Example \ref{ex-equiv} the $(6,4)$-skeleton of $\Delta^1$
is $\{-4,2,5,11\}\cup \{0,4,8,9,13,17\}$ whereas its $(4,6)$-skeleton
is $\{-6,-2,2,3,7,11\} \cup \{0,6,9,15\}$, pictured
as 
\begin{gather*}
\begin{tabular}{cccccccccc}
{\color{red}$\square$}&
{\color{red}$\times$}&
{\color{red}$\times$}&
{\color{red}$\square$}&
{\color{blue}$\square$}&
{\color{red}$\times$}&
{\color{blue}$\square$}&
{\color{blue}$\times$}&
{\color{blue}$\square$}&
{\color{blue}$\times$}
\end{tabular},
\end{gather*}
 which one should compare to \eqref{ex-equiv-skeleton}.

\end{remark}

In other words, the  $(N,M)$-skeleton and the $(M,N)$-skeleton of $\Delta$ differ by a shift by $(M-N)$ which does not depend on 
$\Delta$. In particular, all the constructions are symmetric (up to an overall shift) in $M$ and $N$. From now on we will continue to use notation as in Definitions
\ref{def:skeleton} above and \ref{def:equiv} below.

Let $\Delta$ be an $(N,M)$-invariant set, let $S$ be its skeleton, and $S_0,\ldots, S_{d-1}$ be the parts of the skeleton in different remainders modulo $d=gcd(M,N)$ (i.e., $S_i=S\cap (d\Z+i))$. 

\begin{definition}
\label{def:shift}
A shift $S_0,S_1+a_1,\ldots, S_{d-1}+a_{d-1}$ is called {\it acceptable}
(relative to $S$) if there exists a continuous path $\overline{\phi}=(\phi_1,\ldots,\phi_{d-1}):[0,1]\to\mathbb R^{d-1}$ with  $\overline{\phi}(0)=(0,\ldots,0)$ and $\overline{\phi}(1)=(a_1,\ldots,a_{d-1}),$ such that for any $0\le t\le 1$ the sets $S_0,S_1+\phi_1(t),\ldots, S_{d-1}+\phi_{d-1}(t)$ are pairwise disjoint. In other words, we allow $S_1,\ldots,S_{d-1}$ to shift by translations as long as the elements of different $S_i$'s do not intersect.
In this case we will call the tuple $(a_1,\ldots,a_{d-1})$ an 
acceptable shifting of $S$ and an integral shifting when all $a_i \in
\BZ$.
\end{definition}
When $S$ is understood we will lighten notation by not   specifying the shifts are relative to $S$
or the shiftings are of $S$.
In fact, in the rest of this section we will assume everything
is relative to some fixed skeleton $S$ unless stated otherwise.

\begin{definition}
\label{def:equiv}
Let $\Delta, \Delta' \in \MNM$
with skeletons $S=\bigsqcup S_i$, $S'=\bigsqcup S'_i$, respectively. 
 We say that $\Delta$ is equivalent to $\Delta'$ or 
$\Delta \sim \Delta'$ if there exist a permutation $\sigma\in \Sym_{d-1}$ such that $S'_0,S'_{\sigma(1)},\ldots,S'_{\sigma(d-1)}$ is an acceptable shift of $S_0,S_1,\ldots,S_{d-1}.$  
\end{definition}

Equivalence class will always mean $\sim$ equivalence class.

Dealing with equivalence classes is complicated. Instead, we want to choose one representative from each class---a minimal one, as defined below. Our goal is to shift $S_1,\ldots,S_{d-1}$ down as much as possible, so that the parts of the skeleton are ``stuck" on one another. In fact, this will minimize the size of $\ZZ \setminus \Delta$.  Note that not every integral acceptable shift of a skeleton is again a skeleton of some $(N,M)$-invariant subset, because different parts of the skeleton might end up in the same congruence class modulo $d.$ However, it is convenient to consider the set of all acceptable integral shifts. We will show that there always exists a minimal integral acceptable shift and use it as an intermediate step in the construction of the bijection $\CD:\MNM/\mathord\sim \to \YNM.$ 

\begin{lemma}
The acceptability condition
on a shifting $a_1,\ldots,a_{d-1}$
is equivalent to satisfying a system of
linear inequalities
of the form 
$$
a_i-a_j< \widetilde{b}_{ij},
$$
where $\widetilde{b}_{ij}\in \Z_{>0}\cup\infty$ for $0\le i,j < d,$
are fixed  and the condition $a_0=0$. In particular, the set of acceptable shiftings is convex. 
\end{lemma}

\begin{proof}
Fix a skeleton $S$.
Set 
\begin{equation}
\label{def bij}
\widetilde{b}_{ij}:=\min\limits_{x\in S_i,\ y\in S_j,\ y>x} y-x,
\end{equation}
if $\{x,y:x\in S_i,\ y\in S_j,\ y>x\}\neq\emptyset,$ and $\widetilde{b}_{ij}=\infty$ otherwise. Suppose that $(a_1,\ldots,a_{d-1})$ is an acceptable shifting of $S$. It follows that for any $i,j$ one has $a_i-a_j<\widetilde{b}_{ij}.$ Indeed, assume otherwise, i.e., $a_i-a_j\ge\widetilde{b}_{ij}.$ By definition, there exist $x\in S_i$ and $y\in S_j$ such that $y-x=\widetilde{b}_{ij}>0.$ However, after shifting one has 
$$
(y+a_j)-(x+a_i)=\widetilde{b}_{ij}-(a_i-a_j)\le 0.
$$ 
Therefore, for any continuous path $\overline{\phi}:[0,1]\to \mathbb R^{d-1}$ such that  $\overline{\phi}(0)=(0,\ldots,0)$ and $\overline{\phi}(1)=(a_1,\ldots,a_{d-1}),$ there exists $t\in (0,1]$ such that $y+\phi_i(t)=x+\phi_j(t).$ Contradiction.

Conversely, suppose that $(a_1,\ldots,a_{d-1})$ is such that all the inequalities $a_i-a_j<\widetilde{b}_{ij}$ are satisfied. Take the path $\overline{\phi}$ to be the line segment connecting $(0,\ldots,0)$ and $(a_1,\ldots,a_{d-1}),$ i.e., take
$$
\overline{\phi}(t):=(ta_1,\ldots,ta_{d-1}).
$$
Then for any $t\in [0,1],$ any $1\le i,j\le d-1,\ i\neq j,$ and any $x\in S_i$ and $y\in S_j$ with $y>x$ one has $x+ta_i\neq y+ta_j.$ Indeed,
$$
(y+ta_j)-(x+ta_i)=(y-x)-t(a_i-a_j)>\widetilde{b}_{ij}-t\widetilde{b}_{ij}\ge 0.
$$
(The case $y<x$ is covered similarly by the inequality $a_j-a_i<\widetilde{b}_{ji}.$)
\end{proof}

We are interested in integral acceptable shiftings, so we can set $b_{ij}=\widetilde{b}_{ij}-1,$ and then all integral acceptable shiftings satisfy
\begin{equation}
\label{polytope}
a_i-a_j\le b_{ij}.
\end{equation}
Let $A = A_S \subset \mathbb{R}^{d-1}$ be the set defined by the  inequalities \eqref{polytope}.

\begin{example}\label{example: integral shifts}
Let $(n,m)=(3,2)$ and $d=4.$ Consider the $(12,8)$-invariant subset 
$$
\Delta=\{0,1,5,8,9,12,13,16,17,20,21,24,25,27,28,29,30\}\cup(\Z_{\ge 32}).
$$ 
As in Example \ref{ex-equiv}, it is convenient to reserve a separate row for each remainder modulo $d=4:$ 

\begin{center}
\begin{tabularx}{0.75\textwidth}{XYYYYYYYYYYYYYYYY}
&$-8$&$-4$&$0$&$4$&$8$&$12$&$16$&$20$&$24$&$28$&$32$&$36$&$40$&$44$&$\dots$ \\
{\color{blue}$0$}&{\color{blue}$\square$}&{\color{blue}$\circ$}&{\color{blue}$\times$}&{\color{blue}$\square$}&{\color{blue}$\times$}&{\color{blue}$\bullet$}&{\color{blue}$\times$}&{\color{blue}$\bullet$}&{\color{blue}$\bullet$}&{\color{blue}$\bullet$}&{\color{blue}$\bullet$}&{\color{blue}$\bullet$}&{\color{blue}$\bullet$}&{\color{blue}$\bullet$}&{\color{blue}$\dots$}\\
{\color{mygreen}$1$}&{\color{mygreen}$\square$}&{\color{mygreen}$\square$}&{\color{mygreen}$\times$}&{\color{mygreen}$\times$}&{\color{mygreen}$\times$}&{\color{mygreen}$\bullet$}&{\color{mygreen}$\bullet$}&{\color{mygreen}$\bullet$}&{\color{mygreen}$\bullet$}&{\color{mygreen}$\bullet$}&{\color{mygreen}$\bullet$}&{\color{mygreen}$\bullet$}&{\color{mygreen}$\bullet$}&{\color{mygreen}$\bullet$}&{\color{mygreen}$\dots$}\\
{\color{red}$2$}&{\color{red}$\circ$}&{\color{red}$\circ$}&{\color{red}$\circ$}&{\color{red}$\circ$}&{\color{red}$\circ$}&{\color{red}$\circ$}&{\color{red}$\circ$}&{\color{red}$\square$}&{\color{red}$\square$}&{\color{red}$\times$}&{\color{red}$\times$}&{\color{red}$\times$}&{\color{red}$\bullet$}&{\color{red}$\bullet$}&{\color{red}$\dots$}\\
{\color{orange}$3$}&{\color{orange}$\circ$}&{\color{orange}$\circ$}&{\color{orange}$\circ$}&{\color{orange}$\circ$}&{\color{orange}$\circ$}&{\color{orange}$\circ$}&{\color{orange}$\square$}&{\color{orange}$\circ$}&{\color{orange}$\times$}&{\color{orange}$\square$}&{\color{orange}$\times$}&{\color{orange}$\bullet$}&{\color{orange}$\times$}&{\color{orange}$\bullet$}&{\color{orange}$\dots$}\\
\end{tabularx}
\end{center}

Here each box in the table correspond to the sum of the numbers at the top of the column and at the left and of the row, so to recover the subset $\Delta$ one should read the table column by column, top to bottom, then  left to right. As usual, $\times$ denotes a $12$-generator, and $\square$ an $8$-cogenerator, while $\bullet$ are the other elements of $\Delta$ and $\circ$ the other elements in the complement. The parts $S_0,\ S_1,\ S_2,$ and $S_3$ of the skeleton of $\Delta$ are given by
\begin{gather*}
S_0=\{-8,0,4,8,16\},\
S_1=\{-7,-3,1,5,9\},\
S_2=\{22,26,30,34,38\},
\\
\text{ and  }
S_3=\{19,27,31,35,43\}.
\end{gather*}
We can compute the numbers $b_{ij}=\widetilde{b}_{ij}-1$, $i\neq j$ in this example:
$$
(b_{ij})_{i,j=0}^{3}=
\left(\begin{matrix}
& 0& 5& 2\\
 2& & 12& 9\\
\infty & \infty & & 0\\
\infty & \infty & 2&\\
\end{matrix}\right).
$$
Therefore, the set $A_S\subset \mathbb{R}^3$ is given by
the shiftings $(a_1, a_2, a_3)$ satisfying
$$
-2\le a_0-a_1\le 0,\
-\infty\le a_0-a_2\le 5,\
-\infty\le a_0-a_3\le 2,
$$
$$
-\infty\le a_1-a_2\le 12,\
-\infty\le a_1-a_3\le 9,\
-2\le a_2-a_3\le 0,
$$
where $a_0=0.$ This simplifies to
$$
0\le a_1\le 2,\ a_3\ge -2,\ 0\le a_3-a_2\le 2.
$$
\end{example}

\begin{lemma}\label{Lemma: minimality}
Define
\begin{equation}
\label{def mi}
m_i=\max_{i=i_1,i_2,...,i_k=0} \sum\limits_{\ell=1}^{k-1} (-b_{i_{\ell+1}i_l})
\end{equation}
where the maximum is taken over all sequences $\{i_1,i_2,...,i_k\}$ of integers between $0$ and $d-1,$ such that $i_1=i$ and $i_k=0.$
Then $(m_1,\ldots,m_{d-1})\in A$ and for any $i$ and any integral acceptable shifting $(a_1,\ldots,a_{d-1})\in A\cap\Z^{d-1}$  we have $a_i\ge m_i.$
\end{lemma}

\begin{proof}
By definition, there exists a sequence of integers $i=i_1,i_2,\ldots,i_{k-1}, i_k=0$ such that $m_i=\sum\limits_{\ell=1}^{k-1} (-b_{i_{\ell+1}i_\ell}).$ Then one has
$$
a_i=(a_{i_1}-a_{i_2})+(a_{i_2}-a_{i_3})+\ldots+(a_{i_{k-1}}-a_0)\ge -b_{i_2i_1}-\ldots-b_{i_ki_{k-1}}= m_i.
$$
Suppose $(m_1,\ldots,m_{d-1})\not\in A$. Then $m_i-m_j> b_{ij}$ for some $0\le i,j<d.$ By definition, there exists a sequence $i=i_1,i_2,\ldots, i_k=0$ such that $m_i=\sum\limits_{\ell=1}^{k-1} (-b_{i_{\ell+1}i_l}).$ But then 
$$
m_j < m_i-b_{ij}= - b_{ij}+\sum\limits_{\ell=1}^{k-1} (-b_{i_{\ell+1}i_\ell}),
$$
which contradicts the maximality of $m_j$ (consider the sequence $j,i=i_1,i_2,\ldots, i_k=0$).
In other words, $(m_1,\ldots,m_{d-1})$ is the minimal integral 
acceptable shifting.
\end{proof}

Note all $m_i \le 0$ as $(0,\ldots,0) \in A$. Set $M_0=S_0,M_1=S_1+m_1, \ldots, M_{d-1}=S_{d-1}+m_{d-1}$ to be the shifted parts of the skeleton corresponding to the minimal integral acceptable shift relative to $S$. 

\begin{definition}
Let $f(i)$ be the remainder of any element of $M_i$ modulo $d$ (recall that all elements of $M_i$ have the same remainder).  This defines a function $f:\{0,\ldots,d-1\}\to \{0,\ldots,d-1\}.$ 
\end{definition}

For every $0\le i<d$ set $s_i:=\lfloor\frac{M_i}{d}\rfloor=
\{ \lfloor\frac{x}{d}\rfloor \mid x \in M_i \}$
and let $\Delta_i$ be the $(n,m)$-invariant subset such that $s_i$ is the skeleton of $\Delta_i.$ 
Note $\Delta_i $ might not be $0$-normalized. 
\begin{definition}
Let the directed graph (digraph) $G = G_S$ on the vertex set $\{0,\ldots,d-1\}$ be defined in the following way: vertices $i$ and $j$ are connected by an edge $i\to j$ if $f(i)<f(j)$ and the intersection $s_i\cap s_j$ is not empty. 
\end{definition}

\begin{lemma}\label{Lemma: function from orientation}
The value $f(i)$ equals  the length of the longest oriented path from $0$ to $i$ in the digraph $G.$
\end{lemma}

\begin{proof}
By definition, we have $f(i)<f(j)$ for any edge $i\to j.$ Therefore, it suffices to prove the following two conditions:
\begin{enumerate}
\item If $f(j)>0$ then there exists $i$ such that $f(i)=f(j)-1$ and $G$ contains the edge $i\to j.$
\item $f(i)=0$ implies $i=0.$
\end{enumerate}
Both conditions follow immediately from the minimality of the shift. Indeed, if the first property is not satisfied for a vertex $i$ of $G$ then $(m_1,\ldots,m_i-1,\ldots, m_{d-1})$ is an acceptable shifting, which contradicts Lemma \ref{Lemma: minimality}. 

Suppose now that $f(i)=0,\ i\neq 0.$ This and the first property imply that $f(j)\neq d-1$ for any $j\in\{0,\ldots,d-1\}.$ Therefore, again, $(m_1,\ldots,m_i-1,\ldots, m_{d-1})$ is an acceptable shifting. Contradiction.   
\end{proof}

\begin{corollary}
The function $f$ can be recovered from the orientation of the graph $G.$
\end{corollary}


\begin{example}\label{Example: minimal shift}
Continuing Example \ref{example: integral shifts}, we compute 
$$
m_1=-b_{01}=0,\ m_2=-b_{32}-b_{03}=-4,\ m_3=-b_{03}=-2,
$$
so the \miasing\/ of $S$ is $(0,-4,-3)$.
The minimal integral acceptable shift is then given by
$$
M_0=S_0=\{-8,0,4,8,16\},\
M_1=S_1+0=\{-7,-3,1,5,9\},
$$
$$
M_2=S_2-4=\{18,22,26,30,34\},\
M_3=S_3-2=\{17,25,29,33,41\}.
$$
Note that elements of both $M_1$ and $M_3$ have remainder $1$ modulo $d=4.$ Therefore, this shift does not correspond to any $(12,8)$-invariant subset. 
The skeletons $s_i=\lfloor \frac{M_i}{4}\rfloor$ are given by
$$
s_0=\{-2,0,1,2,4\},\
s_1=\{-2,-1,0,1,2\},
$$
$$
s_2=\{4,5,6,7,8\},\
s_3=\{4,6,7,8,10\},
$$
and we get $f(0)=0,\ f(1)=f(3)=1,$ and $f(2)=2.$
Note $M_i = d s_i + f(i)$. See Figure \ref{Figure-example: graph} for the graph $G.$ We will also consider the $(n,m)$-periodic lattice paths corresponding to $s_0,s_1,s_2,s_3$ (see Figure \ref{Figure: 4 periodic paths}).

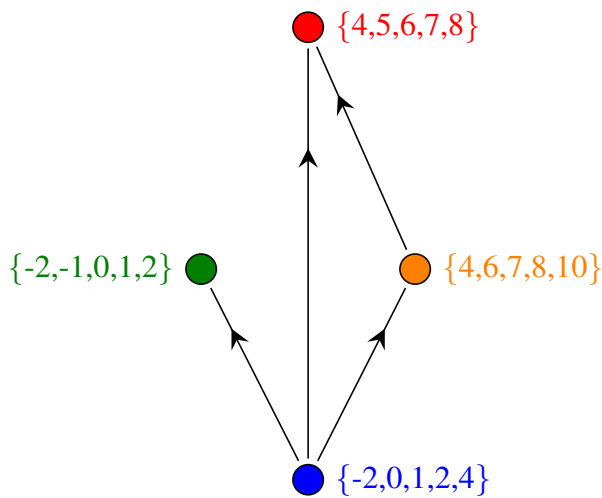
\begin{figure}
\begin{center}
%
%
%
%

\begin{tikzpicture}[
        > = stealth, 
            shorten > = 2pt, 
            shorten < = 2pt, 
            auto,
            semithick 
        ]
\tikzset{myptr/.style={decoration={markings,mark=at position .75 with
    {\arrow[scale=2,>=stealth]{>}}},postaction={decorate}}}
\tikzset{state/.style={circle,draw,minimum size=4mm, inner sep=0pt}, }

    \node[state, fill=red, label=right:{\color{red} \{4,5,6,7,8\}}] (3) {};
    \node[state,fill=orange, label=right: {\color{orange} 
\{4,6,7,8,10\}}] (2) [below right  of=3, xshift=7mm, yshift=-25mm]  {};
    \node[state, fill=mygreen, label=left:{\color{mygreen} 
\{-2,-1,0,1,2\}}] (1) [below left  of=3, xshift=-7mm, yshift=-25mm] {};
    \node[state,fill=blue, label=right:{\color{blue}
\{-2,0,1,2,4\}}] (0) [below of=3, yshift=-50mm] {};

    \draw[myptr] (2) -- (3);
    \draw[myptr] (0) -- (3);
    \draw[myptr] (0) -- (2);
    \draw[myptr] (0) -- (1);

\end{tikzpicture}
\end{center}
\caption{The digraph $G$ with vertices labelled by the skeletons of the $(3,2)$-invariant subsets. The function $f$ corresponds to the levels: $0$ at the blue vertex, $1$ at the green and orange vertices, and $2$ at the red vertex.\label{Figure-example: graph}}
\end{figure}

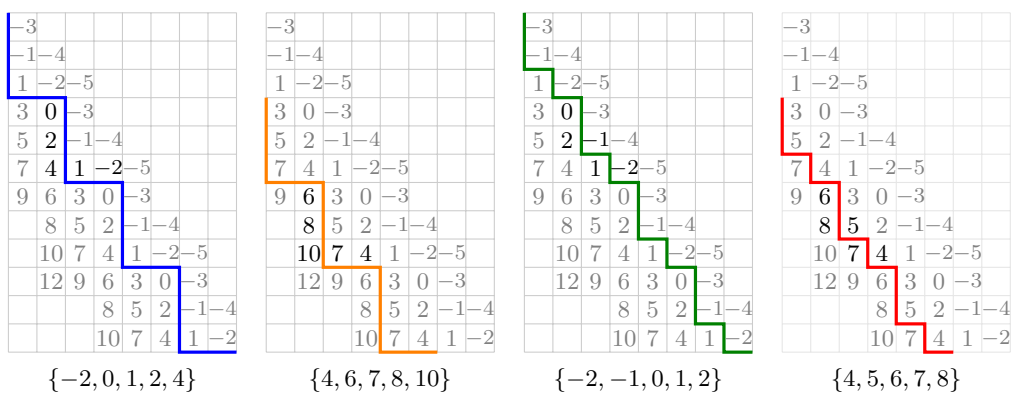
\begin{figure}
\begin{center}
\begin{tikzpicture}[scale=0.375]
\draw [step=1, thin,gray!40!white] (0,0) grid (8,12);


\draw (4,-1) node {\scriptsize{$\{-2,0,1,2,4\}$}};

\draw[blue,very thick] (0,12)--(0,9)--(2,9)--(2,6)--(4,6)--(4,3)--(6,3)--(6,0)--(8,0);

\draw[gray] (.5,11.5) node {\tiny{$-3$}};
\draw[gray] (.5,10.5) node {\tiny{$-1$}};
\draw[gray] (.5,9.5) node {\scriptsize{$1$}};
\draw[gray] (.5,8.5) node {\scriptsize{$3$}};
\draw[gray] (.5,7.5) node {\scriptsize{$5$}};
\draw[gray] (.5,6.5) node {\scriptsize{$7$}};
\draw[gray] (.5,5.5) node {\scriptsize{$9$}};

\draw[gray] (1.5,10.5) node {\tiny{$-4$}};
\draw[gray] (1.5,9.5) node {\tiny{$-2$}};
\draw[black] (1.5,8.5) node {\bf\scriptsize{$0$}};
\draw[black] (1.5,7.5) node {\bf\scriptsize{$2$}};
\draw[black] (1.5,6.5) node {\bf\scriptsize{$4$}};
\draw[gray] (1.5,5.5) node {\scriptsize{$6$}};
\draw[gray] (1.5,4.5) node {\scriptsize{$8$}};
\draw[gray] (1.5,3.5) node {\scriptsize{$10$}};
\draw[gray] (1.5,2.5) node {\scriptsize{$12$}};

\draw[gray] (2.5,9.5) node {\tiny{$-5$}};
\draw[gray] (2.5,8.5) node {\tiny{$-3$}};
\draw[gray] (2.5,7.5) node {\tiny{$-1$}};
\draw[black] (2.5,6.5) node {\bf\scriptsize{$1$}};
\draw[gray] (2.5,5.5) node {\scriptsize{$3$}};
\draw[gray] (2.5,4.5) node {\scriptsize{$5$}};
\draw[gray] (2.5,3.5) node {\scriptsize{$7$}};
\draw[gray] (2.5,2.5) node {\scriptsize{$9$}};

\draw[gray] (3.5,7.5) node {\tiny{$-4$}};
\draw[black] (3.5,6.5) node {\bf\tiny{$-2$}};
\draw[gray] (3.5,5.5) node {\scriptsize{$0$}};
\draw[gray] (3.5,4.5) node {\scriptsize{$2$}};
\draw[gray] (3.5,3.5) node {\scriptsize{$4$}};
\draw[gray] (3.5,2.5) node {\scriptsize{$6$}};
\draw[gray] (3.5,1.5) node {\scriptsize{$8$}};
\draw[gray] (3.5,0.5) node {\scriptsize{$10$}};

\draw[gray] (4.5,6.5) node {\tiny{$-5$}};
\draw[gray] (4.5,5.5) node {\tiny{$-3$}};
\draw[gray] (4.5,4.5) node {\tiny{$-1$}};
\draw[gray] (4.5,3.5) node {\scriptsize{$1$}};
\draw[gray] (4.5,2.5) node {\scriptsize{$3$}};
\draw[gray] (4.5,1.5) node {\scriptsize{$5$}};
\draw[gray] (4.5,.5) node {\scriptsize{$7$}};

\draw[gray] (5.5,4.5) node {\tiny{$-4$}};
\draw[gray] (5.5,3.5) node {\tiny{$-2$}};
\draw[gray] (5.5,2.5) node {\scriptsize{$0$}};
\draw[gray] (5.5,1.5) node {\scriptsize{$2$}};
\draw[gray] (5.5,0.5) node {\scriptsize{$4$}};

\draw[gray] (6.5,3.5) node {\tiny{$-5$}};
\draw[gray] (6.5,2.5) node {\tiny{$-3$}};
\draw[gray] (6.5,1.5) node {\tiny{$-1$}};
\draw[gray] (6.5,.5) node {\scriptsize{$1$}};

\draw[gray] (7.5,1.5) node {\tiny{$-4$}};
\draw[gray] (7.5,0.5) node {\tiny{$-2$}};

\end{tikzpicture}
\begin{tikzpicture}[scale=0.375]
\draw [step=1, gray!40!white] (0,0) grid (8,12);


\draw (4,-1) node {\scriptsize{$\{4,6,7,8,10\}$}};

\draw[orange,very thick] (0,9)--(0,6)--(2,6)--(2,3)--(4,3)--(4,0)--(6,0); 

\draw[gray] (.5,11.5) node {\tiny{$-3$}};
\draw[gray] (.5,10.5) node {\tiny{$-1$}};
\draw[gray] (.5,9.5) node {\scriptsize{$1$}};
\draw[gray] (.5,8.5) node {\scriptsize{$3$}};
\draw[gray] (.5,7.5) node {\scriptsize{$5$}};
\draw[gray] (.5,6.5) node {\scriptsize{$7$}};
\draw[gray] (.5,5.5) node {\scriptsize{$9$}};

\draw[gray] (1.5,10.5) node {\tiny{$-4$}};
\draw[gray] (1.5,9.5) node {\tiny{$-2$}};
\draw[gray] (1.5,8.5) node {\scriptsize{$0$}};
\draw[gray] (1.5,7.5) node {\scriptsize{$2$}};
\draw[gray] (1.5,6.5) node {\scriptsize{$4$}};
\draw[black] (1.5,5.5) node {\bf\scriptsize{$6$}};
\draw[black] (1.5,4.5) node {\bf\scriptsize{$8$}};
\draw[black] (1.5,3.5) node {\bf\scriptsize{$10$}};
\draw[gray] (1.5,2.5) node {\scriptsize{$12$}};

\draw[gray] (2.5,9.5) node {\tiny{$-5$}};
\draw[gray] (2.5,8.5) node {\tiny{$-3$}};
\draw[gray] (2.5,7.5) node {\tiny{$-1$}};
\draw[gray] (2.5,6.5) node {\scriptsize{$1$}};
\draw[gray] (2.5,5.5) node {\scriptsize{$3$}};
\draw[gray] (2.5,4.5) node {\scriptsize{$5$}};
\draw[black] (2.5,3.5) node {\bf\scriptsize{$7$}};
\draw[gray] (2.5,2.5) node {\scriptsize{$9$}};

\draw[gray] (3.5,7.5) node {\tiny{$-4$}};
\draw[gray] (3.5,6.5) node {\tiny{$-2$}};
\draw[gray] (3.5,5.5) node {\scriptsize{$0$}};
\draw[gray] (3.5,4.5) node {\scriptsize{$2$}};
\draw[black] (3.5,3.5) node {\bf\scriptsize{$4$}};
\draw[gray] (3.5,2.5) node {\scriptsize{$6$}};
\draw[gray] (3.5,1.5) node {\scriptsize{$8$}};
\draw[gray] (3.5,0.5) node {\scriptsize{$10$}};

\draw[gray] (4.5,6.5) node {\tiny{$-5$}};
\draw[gray] (4.5,5.5) node {\tiny{$-3$}};
\draw[gray] (4.5,4.5) node {\tiny{$-1$}};
\draw[gray] (4.5,3.5) node {\scriptsize{$1$}};
\draw[gray] (4.5,2.5) node {\scriptsize{$3$}};
\draw[gray] (4.5,1.5) node {\scriptsize{$5$}};
\draw[gray] (4.5,.5) node {\scriptsize{$7$}};

\draw[gray] (5.5,4.5) node {\tiny{$-4$}};
\draw[gray] (5.5,3.5) node {\tiny{$-2$}};
\draw[gray] (5.5,2.5) node {\scriptsize{$0$}};
\draw[gray] (5.5,1.5) node {\scriptsize{$2$}};
\draw[gray] (5.5,0.5) node {\scriptsize{$4$}};

\draw[gray] (6.5,3.5) node {\tiny{$-5$}};
\draw[gray] (6.5,2.5) node {\tiny{$-3$}};
\draw[gray] (6.5,1.5) node {\tiny{$-1$}};
\draw[gray] (6.5,.5) node {\scriptsize{$1$}};

\draw[gray] (7.5,1.5) node {\tiny{$-4$}};
\draw[gray] (7.5,0.5) node {\tiny{$-2$}};

\end{tikzpicture}
\begin{tikzpicture}[scale=0.375]
\draw [step=1, gray!40!white] (0,0) grid (8,12);


\draw (4,-1) node {\scriptsize{$\{-2,-1,0,1,2\}$}};

\draw[mygreen,very thick] (0,12)--(0,10)--(1,10)--(1,9)--(2,9)--(2,7)--(3,7)--(3,6)--(4,6)--(4,6)--(4,4)--(5,4)--(5,3)--(6,3)--(6,1)--(7,1)--(7,0)--(8,0);

\draw[gray] (.5,11.5) node {\tiny{$-3$}};
\draw[gray] (.5,10.5) node {\tiny{$-1$}};
\draw[gray] (.5,9.5) node {\scriptsize{$1$}};
\draw[gray] (.5,8.5) node {\scriptsize{$3$}};
\draw[gray] (.5,7.5) node {\scriptsize{$5$}};
\draw[gray] (.5,6.5) node {\scriptsize{$7$}};
\draw[gray] (.5,5.5) node {\scriptsize{$9$}};

\draw[gray] (1.5,10.5) node {\tiny{$-4$}};
\draw[gray] (1.5,9.5) node {\tiny{$-2$}};
\draw[black] (1.5,8.5) node {\bf\scriptsize{$0$}};
\draw[black] (1.5,7.5) node {\bf\scriptsize{$2$}};
\draw[gray] (1.5,6.5) node {\scriptsize{$4$}};
\draw[gray] (1.5,5.5) node {\scriptsize{$6$}};
\draw[gray] (1.5,4.5) node {\scriptsize{$8$}};
\draw[gray] (1.5,3.5) node {\scriptsize{$10$}};
\draw[gray] (1.5,2.5) node {\scriptsize{$12$}};

\draw[gray] (2.5,9.5) node {\tiny{$-5$}};
\draw[gray] (2.5,8.5) node {\tiny{$-3$}};
\draw[black] (2.5,7.5) node {\bf\tiny{$-1$}};
\draw[black] (2.5,6.5) node {\bf\scriptsize{$1$}};
\draw[gray] (2.5,5.5) node {\scriptsize{$3$}};
\draw[gray] (2.5,4.5) node {\scriptsize{$5$}};
\draw[gray] (2.5,3.5) node {\scriptsize{$7$}};
\draw[gray] (2.5,2.5) node {\scriptsize{$9$}};

\draw[gray] (3.5,7.5) node {\tiny{$-4$}};
\draw[black] (3.5,6.5) node {\bf\tiny{$-2$}};
\draw[gray] (3.5,5.5) node {\scriptsize{$0$}};
\draw[gray] (3.5,4.5) node {\scriptsize{$2$}};
\draw[gray] (3.5,3.5) node {\scriptsize{$4$}};
\draw[gray] (3.5,2.5) node {\scriptsize{$6$}};
\draw[gray] (3.5,1.5) node {\scriptsize{$8$}};
\draw[gray] (3.5,0.5) node {\scriptsize{$10$}};

\draw[gray] (4.5,6.5) node {\tiny{$-5$}};
\draw[gray] (4.5,5.5) node {\tiny{$-3$}};
\draw[gray] (4.5,4.5) node {\tiny{$-1$}};
\draw[gray] (4.5,3.5) node {\scriptsize{$1$}};
\draw[gray] (4.5,2.5) node {\scriptsize{$3$}};
\draw[gray] (4.5,1.5) node {\scriptsize{$5$}};
\draw[gray] (4.5,.5) node {\scriptsize{$7$}};

\draw[gray] (5.5,4.5) node {\tiny{$-4$}};
\draw[gray] (5.5,3.5) node {\tiny{$-2$}};
\draw[gray] (5.5,2.5) node {\scriptsize{$0$}};
\draw[gray] (5.5,1.5) node {\scriptsize{$2$}};
\draw[gray] (5.5,0.5) node {\scriptsize{$4$}};

\draw[gray] (6.5,3.5) node {\tiny{$-5$}};
\draw[gray] (6.5,2.5) node {\tiny{$-3$}};
\draw[gray] (6.5,1.5) node {\tiny{$-1$}};
\draw[gray] (6.5,.5) node {\scriptsize{$1$}};

\draw[gray] (7.5,1.5) node {\tiny{$-4$}};
\draw[gray] (7.5,0.5) node {\tiny{$-2$}};

\end{tikzpicture}
\begin{tikzpicture}[scale=0.375]
\draw [step=1, gray!20!white] (0,0) grid (8,12);


\draw (4,-1) node {\scriptsize{$\{4,5,6,7,8\}$}};

\draw[red,very thick] (0,9)--(0,7)--(1,7)--(1,6)--(2,6)--(2,4)--(3,4)--(3,3)--(4,3)--(4,1)--(5,1)--(5,0)--(6,0);

\draw[gray] (.5,11.5) node {\tiny{$-3$}};
\draw[gray] (.5,10.5) node {\tiny{$-1$}};
\draw[gray] (.5,9.5) node {\scriptsize{$1$}};
\draw[gray] (.5,8.5) node {\scriptsize{$3$}};
\draw[gray] (.5,7.5) node {\scriptsize{$5$}};
\draw[gray] (.5,6.5) node {\scriptsize{$7$}};
\draw[gray] (.5,5.5) node {\scriptsize{$9$}};

\draw[gray] (1.5,10.5) node {\tiny{$-4$}};
\draw[gray] (1.5,9.5) node {\tiny{$-2$}};
\draw[gray] (1.5,8.5) node {\scriptsize{$0$}};
\draw[gray] (1.5,7.5) node {\scriptsize{$2$}};
\draw[gray] (1.5,6.5) node {\scriptsize{$4$}};
\draw[black] (1.5,5.5) node {\bf\scriptsize{$6$}};
\draw[black] (1.5,4.5) node {\bf\scriptsize{$8$}};
\draw[gray] (1.5,3.5) node {\scriptsize{$10$}};
\draw[gray] (1.5,2.5) node {\scriptsize{$12$}};

\draw[gray] (2.5,9.5) node {\tiny{$-5$}};
\draw[gray] (2.5,8.5) node {\tiny{$-3$}};
\draw[gray] (2.5,7.5) node {\tiny{$-1$}};
\draw[gray] (2.5,6.5) node {\scriptsize{$1$}};
\draw[gray] (2.5,5.5) node {\scriptsize{$3$}};
\draw[black] (2.5,4.5) node {\bf\scriptsize{$5$}};
\draw[black] (2.5,3.5) node {\bf\scriptsize{$7$}};
\draw[gray] (2.5,2.5) node {\scriptsize{$9$}};

\draw[gray] (3.5,7.5) node {\tiny{$-4$}};
\draw[gray] (3.5,6.5) node {\tiny{$-2$}};
\draw[gray] (3.5,5.5) node {\scriptsize{$0$}};
\draw[gray] (3.5,4.5) node {\scriptsize{$2$}};
\draw[black] (3.5,3.5) node {\bf\scriptsize{$4$}};
\draw[gray] (3.5,2.5) node {\scriptsize{$6$}};
\draw[gray] (3.5,1.5) node {\scriptsize{$8$}};
\draw[gray] (3.5,0.5) node {\scriptsize{$10$}};

\draw[gray] (4.5,6.5) node {\tiny{$-5$}};
\draw[gray] (4.5,5.5) node {\tiny{$-3$}};
\draw[gray] (4.5,4.5) node {\tiny{$-1$}};
\draw[gray] (4.5,3.5) node {\scriptsize{$1$}};
\draw[gray] (4.5,2.5) node {\scriptsize{$3$}};
\draw[gray] (4.5,1.5) node {\scriptsize{$5$}};
\draw[gray] (4.5,.5) node {\scriptsize{$7$}};

\draw[gray] (5.5,4.5) node {\tiny{$-4$}};
\draw[gray] (5.5,3.5) node {\tiny{$-2$}};
\draw[gray] (5.5,2.5) node {\scriptsize{$0$}};
\draw[gray] (5.5,1.5) node {\scriptsize{$2$}};
\draw[gray] (5.5,0.5) node {\scriptsize{$4$}};

\draw[gray] (6.5,3.5) node {\tiny{$-5$}};
\draw[gray] (6.5,2.5) node {\tiny{$-3$}};
\draw[gray] (6.5,1.5) node {\tiny{$-1$}};
\draw[gray] (6.5,.5) node {\scriptsize{$1$}};

\draw[gray] (7.5,1.5) node {\tiny{$-4$}};
\draw[gray] (7.5,0.5) node {\tiny{$-2$}};

\end{tikzpicture}
\end{center}
\caption{The four $(3,2)$-periodic paths corresponding to the skeletons from Figure \ref{Figure-example: graph}. Note that the elements of the skeletons are exactly the ranks of the boxes above the horizontal steps and to the left of the vertical steps. Equivalently, they are the ranks of the steps of the paths.
\label{Figure: 4 periodic paths}}
\end{figure}

\end{example}

\begin{definition}
Let $T_{n,m}^d$ denote the set of acyclically oriented graphs $G$ on $d$ vertices with a unique source $v_0,$ and vertices labeled by skeletons of $(n,m)$-invariant subsets, such that
\begin{enumerate}
\item All labels are non-negatively normalized, and the label of $v_0$ is zero normalized,
\item Two skeletons intersect if and only if the corresponding vertices are connected by an edge. 
\end{enumerate}
Elements of $T_{n,m}^d$ are considered up to label preserving isomorphisms.
\end{definition}

Note that the underlying (not oriented) graph is determined by the $d$-tuple of $(n,m)$-invariant subsets. We will refer to the orientation of the digraph $G$ as the {\em gluing data} on the $d$-tuple of invariant subsets.

The construction above provides a map 
$$A:\MNM/\mathord\sim \to T_{n,m}^d.$$
Moreover, the map is injective by construction, because given a labeled graph $G\in T_{n,m}^d$ one can use Lemma \ref{Lemma: function from orientation} to reconstruct the function $f(i)$ and then recover the sets $M_0,\ldots, M_{d-1}$ by setting $M_i=ds_i+f(i).$ We need to show that this map is also surjective, i.e., show that for any labeled digraph $G\in T_{n,m}^d$ the corresponding sets $M_0,\ldots, M_{d-1}$ form a minimal integral acceptable shift of the skeleton of an $(N,M)$-invariant subset.

Let $G\in T_{n,m}^d$ be a labeled graph. Let $f$ be the function on the vertex set of $G$ constructed as in Lemma \ref{Lemma: function from orientation}, i.e., for a vertex $v,$ $f(v)$ equals to the length of the longest oriented path from the source $v_0$ to $v.$ Since elements of $T_{n,m}^d$ are considered up to label preserving isomorphisms, one can assume that the vertex set of $G$ is $V_G=\{0,1,\ldots, d-1\}$ and the function $f$ is weakly monotone:
$$
i<j \Rightarrow f(i)\le f(j).
$$
In particular, one gets $v_0=0.$ Note, that two different labeled digraphs on the vertex set $\{0,1,\ldots,d-1\}$ might be related by a label preserving isomorphism, in which case they correspond to the same element of $T_{n,m}^d$ (see Figure \ref{Figure: two graphs} for an example).   

Let the vertices $0,\ldots, d-1$ of $G$
have labels $s_0,\ldots, s_{d-1}$ respectively.


\begin{lemma}\label{Lemma: representative}
Given $G \in T_{n,m}^d$,
there exists an $(N,M)$-invariant subset $\Delta = \Delta^G$ with skeleton $S=\bigsqcup S_i,$ such that
$$
M_0:=ds_0,M_1:=ds_1+f(1),\ldots,M_{d-1}:=ds_{d-1}+f(d-1)
$$ 
is the minimal integral acceptable shift of $S_0,\ldots,S_{d-1}.$ More over, we can choose $\Delta$ by setting
$$
S_0:=ds_0,\ S_1:=ds_1+1,\ldots,S_{d-1}:=ds_{d-1}+(d-1).
$$ 
\end{lemma}

\begin{proof}
By construction, for every $i$ every element of $S_i$ has remainder $i$ modulo $d.$ It follows that $S:=\bigsqcup S_i$ is the skeleton of an $(N,M)$-invariant set $\Delta.$ The non-negative normalization of the $s_i$ imply $\Delta \in \MNM$.  It remains to show that $M_0,M_1,\ldots,M_{d-1}$ is the minimal integral acceptable shift of $S_0,S_1,\ldots,S_{d-1}.$ By construction, for every $i$ we have $M_i=S_i+a_i$ where $a_i:=f(i)-i.$ Recall that the minimal integral acceptable shifting is given by \eqref{def mi},
and by \eqref{def bij} the integers $b_{ij}$ are given by 
$$
b_{ij}:=\left(\min\limits_{x\in S_i,\ y\in S_j,\ y>x} y-x\right)-1.
$$ 

Suppose
$i \to j$
is an edge of $G.$ Then as $f$ is monotone we must have $i<j$. Since $s_i\cap s_j\neq\emptyset$ it follows that there are $x\in S_i$ and $y\in S_j$ such that $\lfloor\frac{x}{d}\rfloor=\lfloor\frac{y}{d}\rfloor,$ so we get $b_{ij}=(j-i)-1.$
On the other hand,
suppose there is no directed edge $i \to j$. If $i<j$
we immediately get that $b_{ij}\ge (d+j-i)-1 \ge j.$ Similarly, if $i>j$
(even if $j \to i$) we get $b_{ij}\ge (j+d-i)-1\ge j.$
It follows that to maximize $\sum\limits_{\ell=1}^{k-1}( -b_{i_{\ell+1}i_\ell})$ in \eqref{def mi} 
one has to consider the longest directed path in $G$ from $0$ to $i$. Such a path has $f(i)$ steps, so we get 
$$
m_i=f(i)-i=a_i.
$$
Therefore, $M_0,M_1,\ldots, M_{d-1}$ is the minimal integral acceptable shift of $S_0,S_1,\ldots, S_{d-1}.$
\end{proof}

Note that the representative constructed in Lemma \ref{Lemma: representative} above has the following property: $\left\lfloor\frac{S_i}{d}\right\rfloor=\left\lfloor\frac{M_i}{d}\right\rfloor$ where, as before, $S_0,\ldots, S_{d-1}$ are the parts of the skeleton of the representative $\Delta$ of the equivalence class in the corresponding remainders modulo $d,$ and $M_0,\ldots, M_{d-1}$ is the minimal integral acceptable shift of $S_0,\ldots, S_{d-1}.$ We call such representatives $\Delta$ {\it the minimal representatives}. Note that an equivalence class might contain more than one minimal representative (see Example \ref{Example: minimal reps}).
Recall the sets $M_0,\ldots,M_{d-1}$ might not correspond to an element of $\MNM$, but the $S_0,\ldots,S_{d-1}$ will.

\begin{figure}
\begin{center}
  \begin{tikzpicture}[scale=0.4,
        > = stealth, 
            shorten > = 2pt, 
            shorten < = 2pt, 
            auto,
            semithick 
        ]
\tikzset{myptr/.style={decoration={markings,mark=at position .75 with 
    {\arrow[scale=2,>=stealth]{>}}},postaction={decorate}}}
\tikzset{state/.style={circle,draw,minimum size=6mm, inner sep=0pt}, }

   \begin{scope}[xshift=-8cm]
    \node[state,label=right:{\scriptsize \{4,5,6,7,8\}}] (3) {$3$};
    \node[state,label=right: {\scriptsize \{4,6,7,8,10\}}] (2)
        [below right  of=3, xshift=7mm, yshift=-25mm]  {$1$};
    \node[state,label=left:{\scriptsize \{-2,-1,0,1,2\}}] (1)
        [below left  of=3, xshift=-7mm, yshift=-25mm] {$2$};
    \node[state,label=right:{\scriptsize \{-2,0,1,2,4\}}] (0)
        [below of=3, yshift=-50mm] {$0$};

    \draw[myptr] (2) -- (3);
    \draw[myptr] (0) -- (3);
    \draw[myptr] (0) -- (2);
    \draw[myptr] (0) -- (1);
   \end{scope}

   \begin{scope}[xshift=12cm]
    \node[state,label=right:{\scriptsize \{4,5,6,7,8\}}] (3) {$3$};
    \node[state,label=left: {\scriptsize \{4,6,7,8,10\}}] (2) [below left of=3,
        xshift=-7mm, yshift=-25mm]  {$2$};
    \node[state,label=right:{\scriptsize \{-2,-1,0,1,2\}}] (1) [below right of=3,
        xshift=7mm, yshift=-25mm] {$1$};
    \node[state,label=right:{\scriptsize \{-2,0,1,2,4\}}] (0) [below of=3,
        yshift=-50mm] {$0$};

    \draw[myptr] (2) -- (3);
    \draw[myptr] (0) -- (3);
    \draw[myptr] (0) -- (2);
    \draw[myptr] (0) -- (1);
   \end{scope}

\end{tikzpicture}

\end{center}
\caption{Two graphs corresponding to the same point of $T_{3,2}^4:$ the isomorhism switching vertices $1$ and $2$ preserves the labels. \label{Figure: two graphs}}
\end{figure}
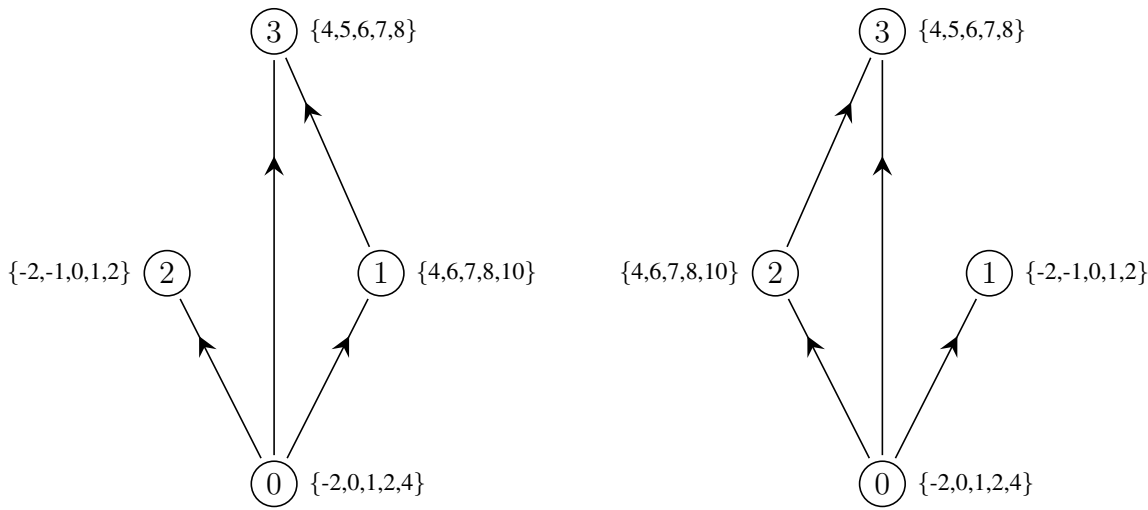

\begin{example}\label{Example: minimal reps}
Continuing Example \ref{Example: minimal shift} and using the graph on the left of the Figure \ref{Figure: two graphs}, one gets
\begin{gather*}
S_0=4\{-2,0,1,2,4\}=\{-8,0,4,8,16\}, \\
S_1=4\{4,6,7,8,10\}+1=\{17,25,29,33,41\}, \\
S_2=4\{-2,-1,0,1,2\}+2=\{-6,-2,0,6,10\}, \\
S_3=4\{4,5,6,7,8\}+3=\{19,23,27,31,35\}.
\end{gather*}
Therefore, the $(12,8)$-invariant subset $\Delta^G$ we constructed
is depicted below. 
\begin{center}
\begin{tabularx}{0.75\textwidth}{XYYYYYYYYYYYYYYYY}
&$-8$&$-4$&$0$&$4$&$8$&$12$&$16$&$20$&$24$&$28$&$32$&$36$&$40$&$44$&$\dots$ \\
{\color{blue}$0$}&{\color{blue}$\square$}&{\color{blue}$\circ$}&{\color{blue}$\times$}&{\color{blue}$\square$}&{\color{blue}$\times$}&{\color{blue}$\bullet$}&{\color{blue}$\times$}&{\color{blue}$\bullet$}&{\color{blue}$\bullet$}&{\color{blue}$\bullet$}&{\color{blue}$\bullet$}&{\color{blue}$\bullet$}&{\color{blue}$\bullet$}&{\color{blue}$\bullet$}&{\color{blue}$\dots$}\\
{\color{red}$1$}&{\color{red}$\circ$}&{\color{red}$\circ$}&{\color{red}$\circ$}&{\color{red}$\circ$}&{\color{red}$\circ$}&{\color{red}$\circ$}&{\color{red}$\square$}&{\color{red}$\circ$}&{\color{red}$\times$}&{\color{red}$\square$}&{\color{red}$\times$}&{\color{red}$\bullet$}&{\color{red}$\times$}&{\color{red}$\bullet$}&{\color{red}$\dots$}\\
{\color{mygreen}$2$}&{\color{mygreen}$\square$}&{\color{mygreen}$\square$}&{\color{mygreen}$\times$}&{\color{mygreen}$\times$}&{\color{mygreen}$\times$}&{\color{mygreen}$\bullet$}&{\color{mygreen}$\bullet$}&{\color{mygreen}$\bullet$}&{\color{mygreen}$\bullet$}&{\color{mygreen}$\bullet$}&{\color{mygreen}$\bullet$}&{\color{mygreen}$\bullet$}&{\color{mygreen}$\bullet$}&{\color{mygreen}$\bullet$}&{\color{mygreen}$\dots$}\\
{\color{orange}$3$}&{\color{orange}$\circ$}&{\color{orange}$\circ$}&{\color{orange}$\circ$}&{\color{orange}$\circ$}&{\color{orange}$\circ$}&{\color{orange}$\circ$}&{\color{orange}$\square$}&{\color{orange}$\square$}&{\color{orange}$\times$}&{\color{orange}$\times$}&{\color{orange}$\times$}&{\color{orange}$\bullet$}&{\color{orange}$\bullet$}&{\color{orange}$\bullet$}&{\color{orange}$\dots$}\\
\end{tabularx}
\end{center}
Note that if we had used the graph on the right of Figure \ref{Figure: two graphs} instead, we would get a different invariant subset in the same equivalence class. Both these subsets are minimal, as they are constructed according to the algorithm in Lemma \ref{Lemma: representative}. Note that they are both different from the invariant subset we started from in Example \ref{example: integral shifts} , which was not minimal. 

Also note this set has $14$ gaps, which is the area of
the rational Dyck path constructed in Figure \ref{Figure: gluing}.
\end{example}

\subsection{From equivalence classes to Dyck paths}
\label{sec-Dyck}

The last step is to construct a bijection $B:T_{n,m}^d\to \YNM$ so that we can set $\CD=B\circ A.$ Let $G\in T_{n,m}^d$ be a labeled graph, and let $P_0,\ldots, P_{d-1}$ be the $(m,n)$-periodic lattice paths corresponding to the labels $s_0,\ldots, s_{d-1}$ of $G.$ 

\begin{lemma}\label{Lemma: intersections}
Let $0\le i<j<d.$ Then the paths $P_i$ and $P_j$ intersect if and only if the skeletons $s_i$ and $s_j$ intersect or, equivalently, if and only if the graph $G$ has an edge between the corresponding vertices.
\end{lemma}

\begin{proof}
Suppose that $x\in s_i\cap s_j,$ and let $\square$ be a box
in $\Z^2$
with $\rank (\square) = x$.
Then $P_i$ contains either the step $\ver_\square$ that is to the right
(if $x$ is an $n$-generator) or the step $\hor_\square$ at its bottom  (if $x$ is an $m$-cogenerator). (Here we extend the notation from Lemma \ref{lemma-dinv} to periodic paths.)
In both cases, $P_i$ passes through the right-bottom corner of $\square$. The  same is true for $P_j.$
Hence $P_i $ intersects $P_j$ at that corner.

Conversely, if $P_i$ contains a lattice point $p$
let $\square$ be the box with $p$ at its bottom right corner.
Then $P_i$ must contain either the step $\ver_\square$
going up from $p$, or the the step $\hor_\square$ going left from $p$.
In both cases, it implies $x = \rank (\square)$ 
has to be in the skeleton $s_i.$
The same holds for $P_j$ and $s_j,$ so $x \in s_i \cap s_j$. 
 Finally, the graph $G$ has an edge between two vertices if and only if the corresponding skeletons intersect.
\end{proof}

\subsubsection{Gluing algorithm}
\label{sec-glue}

We will glue together paths $P_0,\ldots, P_{d-1}$ (more precisely, a
union of possibly disconnected  intervals of total length $(n+m)$ of these paths) to get an $(N,M)$-Dyck path $D$ in the following way,
which we call our gluing algorithm. We start by taking the interval of $P_0$ that is an $(n,m)$-Dyck path (there is a unique way to choose such an interval, up to a periodic shift). At each step we glue in an interval of length $(n+m)$ of one of the periodic paths $P_1,\ldots,P_{d-1}$ into our path. This is done using the following procedure. 

Let $\hat{D}$ be a $(kn,km)$-Dyck path and let $\hat{P}$ be an  $(n,m)$-periodic path, such that the intersection $\hat{D}\cap\hat{P}$ is not empty. Let $p$ be the first point of intersection of $\hat{D}$ and $\hat{P}$ relative to $\hat{D}$ (recall that we orient all lattice paths from bottom-right to top-left, i.e., $p$ is the point of intersection, closest to the bottom-right end of $\hat{D}$). The new $\left((k+1)n,(k+1)m\right)$-Dyck path $\hat{D}\vee\hat{P}$ is the union of three lattice paths: 
\begin{enumerate}
\item First we follow the path $\hat{D}$ from its start up to $p$;
\item Then we follow $\hat{P}$ for $(n+m)$ steps starting at $p$;
\item Finally, we follow the remaining part of $\hat{D}$ translated by $n$ up and $m$ to the left, i.e., by $+(-m,n)$.
\end{enumerate}    

More precisely, let us now also identify a $(kn,km)$-Dyck path 
$\hat D$ with the function $\hat D:[0,k(n+m)] \to  \BR^2$, so that $\hat D$ is its plot and the function is an isometry to the image.
Similarly, a periodic path can be regarded as a function
$P: \BR \to \BR^2$ satisfying $P(z+m+n) = P(z) + (-m,n)$.
Given $r \in \BZ$ and
a function $I: [r, r+n+m] \to \BR^2$ satisfying
$I(r+n+m) = I(r) + (-m,n)$, we extend $I$ periodically to
$P(I):\BR \to \BR^2$ by $P(I)(z +k(n+m)) = I(z) +k(-m,n)$,
for $r\le z\le r+n+m$. 
Note if $I$ was an interval of a $(kn,km)$-Dyck path then 
$P(I)$ is normalized so $0\le z\le n+m$ implies $P(z)$ is
between the lines $y=0$ and $y=n$.  However, it is convenient
to treat periodic paths so their parameterization might not
be normalized in this way (i.e., so that we need not have $a=b$ below, or so that we can consider an interval
of it as a function with domain $[0,n+m]$).
Using this  function notation, we may describe
$$\hat{D}\vee\hat{P}(z) = \begin{cases}
\hat D (z) & 0 \le z \le a\\
\hat P (z +b-a) & a \le z \le a+n+m\\
\hat D (z -(m+n)) + (-m,n) & a+n+m \le z \le (k+1)(n+m)
\end{cases}$$
where $a,b \in \BR$ are the parameters such that $\hat{D} (a) =
\hat{P}(b) = p$ and $p \in \BR^2$ is the  first point of $\hat D$ 
that is also in $\hat P$.

We apply the above procedure $d-1$ times in the following order. Let $k_j = \#\{i \mid f(i) \le j\}$. We start by setting $D_0$ to be the interval of $P_0$ such that $D_0$ is an $(n,m)$-Dyck path. Take all paths $P_i$, such that $f(i)=1.$
Note that all such paths intersect $D_0$ and do not intersect each other. Therefore, we can glue them in using the above procedure, and the order in which we do it does not matter, i.e. the path created is independent of gluing order for these $i$. Let $D_1$ be the resulting rational Dyck path. 
Note is it a $(k_1 n , k_1 m)$-Dyck path.

At the $(j+1)$th step we start with the  $(k_jn,k_jm)$-Dyck path $D_j$ obtained from $D_0$ by gluing in intervals of all paths $P_i$ such that $f(i)\le j,$ one level of $G$ at a time, and we glue in intervals of all $P_i$'s, such that $f(i)=j+1.$ Again, all such paths intersect at least one of the intervals we glued in on the previous step, and they do not intersect each other. We proceed in the same manner until we glued in intervals of all periodic paths $P_1,\ldots, P_{d-1}$. (See Figure \ref{Figure: gluing} for an example.)

\begin{figure}
\begin{center}
\begin{tikzpicture}[scale=0.375]
\draw [step=1, thin,gray!40!white] (0,0) grid (8,12);


\draw[orange!50!white,thick] (0,9)--(0,6)--(2,6)--(2,3)--(4,3)--(4,0)--(6,0);
\draw[blue,very thick] (0,12)--(0,9)--(2,9);

\draw[fill=orange] (0,9) circle (2mm);

\draw[gray] (.5,11.5) node {\tiny{$-3$}};
\draw[gray] (.5,10.5) node {\tiny{$-1$}};
\draw[gray] (.5,9.5) node {\scriptsize{$1$}};
\draw[gray] (.5,8.5) node {\scriptsize{$3$}};
\draw[gray] (.5,7.5) node {\scriptsize{$5$}};
\draw[gray] (.5,6.5) node {\scriptsize{$7$}};
\draw[gray] (.5,5.5) node {\scriptsize{$9$}};

\draw[gray] (1.5,10.5) node {\tiny{$-4$}};
\draw[gray] (1.5,9.5) node {\tiny{$-2$}};
\draw[gray] (1.5,8.5) node {\scriptsize{$0$}};
\draw[gray] (1.5,7.5) node {\scriptsize{$2$}};
\draw[gray] (1.5,6.5) node {\scriptsize{$4$}};
\draw[gray] (1.5,5.5) node {\scriptsize{$6$}};
\draw[gray] (1.5,4.5) node {\scriptsize{$8$}};
\draw[gray] (1.5,3.5) node {\scriptsize{$10$}};
\draw[gray] (1.5,2.5) node {\scriptsize{$12$}};

\draw[gray] (2.5,9.5) node {\tiny{$-5$}};
\draw[gray] (2.5,8.5) node {\tiny{$-3$}};
\draw[gray] (2.5,7.5) node {\tiny{$-1$}};
\draw[gray] (2.5,6.5) node {\scriptsize{$1$}};
\draw[gray] (2.5,5.5) node {\scriptsize{$3$}};
\draw[gray] (2.5,4.5) node {\scriptsize{$5$}};
\draw[gray] (2.5,3.5) node {\scriptsize{$7$}};
\draw[gray] (2.5,2.5) node {\scriptsize{$9$}};

\draw[gray] (3.5,7.5) node {\tiny{$-4$}};
\draw[gray] (3.5,6.5) node {\tiny{$-2$}};
\draw[gray] (3.5,5.5) node {\scriptsize{$0$}};
\draw[gray] (3.5,4.5) node {\scriptsize{$2$}};
\draw[gray] (3.5,3.5) node {\scriptsize{$4$}};
\draw[gray] (3.5,2.5) node {\scriptsize{$6$}};
\draw[gray] (3.5,1.5) node {\scriptsize{$8$}};
\draw[gray] (3.5,0.5) node {\scriptsize{$10$}};

\draw[gray] (4.5,6.5) node {\tiny{$-5$}};
\draw[gray] (4.5,5.5) node {\tiny{$-3$}};
\draw[gray] (4.5,4.5) node {\tiny{$-1$}};
\draw[gray] (4.5,3.5) node {\scriptsize{$1$}};
\draw[gray] (4.5,2.5) node {\scriptsize{$3$}};
\draw[gray] (4.5,1.5) node {\scriptsize{$5$}};
\draw[gray] (4.5,.5) node {\scriptsize{$7$}};

\draw[gray] (5.5,4.5) node {\tiny{$-4$}};
\draw[gray] (5.5,3.5) node {\tiny{$-2$}};
\draw[gray] (5.5,2.5) node {\scriptsize{$0$}};
\draw[gray] (5.5,1.5) node {\scriptsize{$2$}};
\draw[gray] (5.5,0.5) node {\scriptsize{$4$}};

\draw[gray] (6.5,3.5) node {\tiny{$-5$}};
\draw[gray] (6.5,2.5) node {\tiny{$-3$}};
\draw[gray] (6.5,1.5) node {\tiny{$-1$}};
\draw[gray] (6.5,.5) node {\scriptsize{$1$}};

\draw[gray] (7.5,1.5) node {\tiny{$-4$}};
\draw[gray] (7.5,0.5) node {\tiny{$-2$}};

\draw (9,6) node {$\leadsto$};
\end{tikzpicture}
\begin{tikzpicture}[scale=0.375]
\draw [step=1, gray!40!white] (0,0) grid (8,12);


\draw[mygreen!50!white,thick] (0,12)--(0,10)--(1,10)--(1,9)--(2,9)--(2,7)--(3,7)--(3,6)--(4,6)--(4,4)--(5,4)--(5,3)--(6,3)--(6,1)--(7,1)--(7,0)--(8,0);

\draw[blue,very thick] (0,12)--(0,9); 
\draw[orange,very thick] (0,9)--(0,6)--(2,6); 
\draw[blue,very thick] (2,6)--(4,6); 

\draw[fill=mygreen] (4,6) circle (2mm);

\draw[gray] (.5,11.5) node {\tiny{$-3$}};
\draw[gray] (.5,10.5) node {\tiny{$-1$}};
\draw[gray] (.5,9.5) node {\scriptsize{$1$}};
\draw[gray] (.5,8.5) node {\scriptsize{$3$}};
\draw[gray] (.5,7.5) node {\scriptsize{$5$}};
\draw[gray] (.5,6.5) node {\scriptsize{$7$}};
\draw[gray] (.5,5.5) node {\scriptsize{$9$}};

\draw[gray] (1.5,10.5) node {\tiny{$-4$}};
\draw[gray] (1.5,9.5) node {\tiny{$-2$}};
\draw[gray] (1.5,8.5) node {\scriptsize{$0$}};
\draw[gray] (1.5,7.5) node {\scriptsize{$2$}};
\draw[gray] (1.5,6.5) node {\scriptsize{$4$}};
\draw[gray] (1.5,5.5) node {\scriptsize{$6$}};
\draw[gray] (1.5,4.5) node {\scriptsize{$8$}};
\draw[gray] (1.5,3.5) node {\scriptsize{$10$}};
\draw[gray] (1.5,2.5) node {\scriptsize{$12$}};

\draw[gray] (2.5,9.5) node {\tiny{$-5$}};
\draw[gray] (2.5,8.5) node {\tiny{$-3$}};
\draw[gray] (2.5,7.5) node {\tiny{$-1$}};
\draw[gray] (2.5,6.5) node {\scriptsize{$1$}};
\draw[gray] (2.5,5.5) node {\scriptsize{$3$}};
\draw[gray] (2.5,4.5) node {\scriptsize{$5$}};
\draw[gray] (2.5,3.5) node {\scriptsize{$7$}};
\draw[gray] (2.5,2.5) node {\scriptsize{$9$}};

\draw[gray] (3.5,7.5) node {\tiny{$-4$}};
\draw[gray] (3.5,6.5) node {\tiny{$-2$}};
\draw[gray] (3.5,5.5) node {\scriptsize{$0$}};
\draw[gray] (3.5,4.5) node {\scriptsize{$2$}};
\draw[gray] (3.5,3.5) node {\scriptsize{$4$}};
\draw[gray] (3.5,2.5) node {\scriptsize{$6$}};
\draw[gray] (3.5,1.5) node {\scriptsize{$8$}};
\draw[gray] (3.5,0.5) node {\scriptsize{$10$}};

\draw[gray] (4.5,6.5) node {\tiny{$-5$}};
\draw[gray] (4.5,5.5) node {\tiny{$-3$}};
\draw[gray] (4.5,4.5) node {\tiny{$-1$}};
\draw[gray] (4.5,3.5) node {\scriptsize{$1$}};
\draw[gray] (4.5,2.5) node {\scriptsize{$3$}};
\draw[gray] (4.5,1.5) node {\scriptsize{$5$}};
\draw[gray] (4.5,.5) node {\scriptsize{$7$}};

\draw[gray] (5.5,4.5) node {\tiny{$-4$}};
\draw[gray] (5.5,3.5) node {\tiny{$-2$}};
\draw[gray] (5.5,2.5) node {\scriptsize{$0$}};
\draw[gray] (5.5,1.5) node {\scriptsize{$2$}};
\draw[gray] (5.5,0.5) node {\scriptsize{$4$}};

\draw[gray] (6.5,3.5) node {\tiny{$-5$}};
\draw[gray] (6.5,2.5) node {\tiny{$-3$}};
\draw[gray] (6.5,1.5) node {\tiny{$-1$}};
\draw[gray] (6.5,.5) node {\scriptsize{$1$}};

\draw[gray] (7.5,1.5) node {\tiny{$-4$}};
\draw[gray] (7.5,0.5) node {\tiny{$-2$}};

\draw (9,6) node {$\leadsto$};
\end{tikzpicture}
\begin{tikzpicture}[scale=0.375]
\draw [step=1, gray!40!white] (0,0) grid (8,12);


\draw[red!40!white,thick] (0,9)--(0,7)--(1,7)--(1,6)--(2,6)--(2,4)--(3,4)--(3,3)--(4,3)--(4,1)--(5,1)--(5,0)--(6,0)--(6,0);

\draw[blue,very thick] (0,12)--(0,9); 
\draw[orange,very thick] (0,9)--(0,6)--(2,6); 
\draw[blue,very thick] (2,6)--(4,6); 
\draw[mygreen,very thick] (4,6)--(4,4)--(5,4)--(5,3)--(6,3);

\draw[fill=red] (2,6) circle (2mm);

\draw[gray] (.5,11.5) node {\tiny{$-3$}};
\draw[gray] (.5,10.5) node {\tiny{$-1$}};
\draw[gray] (.5,9.5) node {\scriptsize{$1$}};
\draw[gray] (.5,8.5) node {\scriptsize{$3$}};
\draw[gray] (.5,7.5) node {\scriptsize{$5$}};
\draw[gray] (.5,6.5) node {\scriptsize{$7$}};
\draw[gray] (.5,5.5) node {\scriptsize{$9$}};

\draw[gray] (1.5,10.5) node {\tiny{$-4$}};
\draw[gray] (1.5,9.5) node {\tiny{$-2$}};
\draw[gray] (1.5,8.5) node {\scriptsize{$0$}};
\draw[gray] (1.5,7.5) node {\scriptsize{$2$}};
\draw[gray] (1.5,6.5) node {\scriptsize{$4$}};
\draw[gray] (1.5,5.5) node {\scriptsize{$6$}};
\draw[gray] (1.5,4.5) node {\scriptsize{$8$}};
\draw[gray] (1.5,3.5) node {\scriptsize{$10$}};
\draw[gray] (1.5,2.5) node {\scriptsize{$12$}};

\draw[gray] (2.5,9.5) node {\tiny{$-5$}};
\draw[gray] (2.5,8.5) node {\tiny{$-3$}};
\draw[gray] (2.5,7.5) node {\tiny{$-1$}};
\draw[gray] (2.5,6.5) node {\scriptsize{$1$}};
\draw[gray] (2.5,5.5) node {\scriptsize{$3$}};
\draw[gray] (2.5,4.5) node {\scriptsize{$5$}};
\draw[gray] (2.5,3.5) node {\scriptsize{$7$}};
\draw[gray] (2.5,2.5) node {\scriptsize{$9$}};

\draw[gray] (3.5,7.5) node {\tiny{$-4$}};
\draw[gray] (3.5,6.5) node {\tiny{$-2$}};
\draw[gray] (3.5,5.5) node {\scriptsize{$0$}};
\draw[gray] (3.5,4.5) node {\scriptsize{$2$}};
\draw[gray] (3.5,3.5) node {\scriptsize{$4$}};
\draw[gray] (3.5,2.5) node {\scriptsize{$6$}};
\draw[gray] (3.5,1.5) node {\scriptsize{$8$}};
\draw[gray] (3.5,0.5) node {\scriptsize{$10$}};

\draw[gray] (4.5,6.5) node {\tiny{$-5$}};
\draw[gray] (4.5,5.5) node {\tiny{$-3$}};
\draw[gray] (4.5,4.5) node {\tiny{$-1$}};
\draw[gray] (4.5,3.5) node {\scriptsize{$1$}};
\draw[gray] (4.5,2.5) node {\scriptsize{$3$}};
\draw[gray] (4.5,1.5) node {\scriptsize{$5$}};
\draw[gray] (4.5,.5) node {\scriptsize{$7$}};

\draw[gray] (5.5,4.5) node {\tiny{$-4$}};
\draw[gray] (5.5,3.5) node {\tiny{$-2$}};
\draw[gray] (5.5,2.5) node {\scriptsize{$0$}};
\draw[gray] (5.5,1.5) node {\scriptsize{$2$}};
\draw[gray] (5.5,0.5) node {\scriptsize{$4$}};

\draw[gray] (6.5,3.5) node {\tiny{$-5$}};
\draw[gray] (6.5,2.5) node {\tiny{$-3$}};
\draw[gray] (6.5,1.5) node {\tiny{$-1$}};
\draw[gray] (6.5,.5) node {\scriptsize{$1$}};

\draw[gray] (7.5,1.5) node {\tiny{$-4$}};
\draw[gray] (7.5,0.5) node {\tiny{$-2$}};

\draw (9,6) node {$\leadsto$};
\end{tikzpicture}
\begin{tikzpicture}[scale=0.375]
\draw [step=1, gray!20!white] (0,0) grid (8,12);


\draw[blue,very thick] (0,12)--(0,9); 
\draw[orange,very thick] (0,9)--(0,6)--(2,6); 
\draw[red,very thick] (2,6)--(2,4)--(3,4)--(3,3)--(4,3);
\draw[blue,very thick] (4,3)--(6,3); 
\draw[mygreen,very thick] (6,3)--(6,1)--(7,1)--(7,0)--(8,0);

\draw[gray] (.5,11.5) node {\tiny{$-3$}};
\draw[gray] (.5,10.5) node {\tiny{$-1$}};
\draw[gray] (.5,9.5) node {\scriptsize{$1$}};
\draw[gray] (.5,8.5) node {\scriptsize{$3$}};
\draw[gray] (.5,7.5) node {\scriptsize{$5$}};
\draw[gray] (.5,6.5) node {\scriptsize{$7$}};
\draw[gray] (.5,5.5) node {\scriptsize{$9$}};

\draw[gray] (1.5,10.5) node {\tiny{$-4$}};
\draw[gray] (1.5,9.5) node {\tiny{$-2$}};
\draw[gray] (1.5,8.5) node {\scriptsize{$0$}};
\draw[gray] (1.5,7.5) node {\scriptsize{$2$}};
\draw[gray] (1.5,6.5) node {\scriptsize{$4$}};
\draw[gray] (1.5,5.5) node {\scriptsize{$6$}};
\draw[gray] (1.5,4.5) node {\scriptsize{$8$}};
\draw[gray] (1.5,3.5) node {\scriptsize{$10$}};
\draw[gray] (1.5,2.5) node {\scriptsize{$12$}};

\draw[gray] (2.5,9.5) node {\tiny{$-5$}};
\draw[gray] (2.5,8.5) node {\tiny{$-3$}};
\draw[gray] (2.5,7.5) node {\tiny{$-1$}};
\draw[gray] (2.5,6.5) node {\scriptsize{$1$}};
\draw[gray] (2.5,5.5) node {\scriptsize{$3$}};
\draw[gray] (2.5,4.5) node {\scriptsize{$5$}};
\draw[gray] (2.5,3.5) node {\scriptsize{$7$}};
\draw[gray] (2.5,2.5) node {\scriptsize{$9$}};

\draw[gray] (3.5,7.5) node {\tiny{$-4$}};
\draw[gray] (3.5,6.5) node {\tiny{$-2$}};
\draw[gray] (3.5,5.5) node {\scriptsize{$0$}};
\draw[gray] (3.5,4.5) node {\scriptsize{$2$}};
\draw[gray] (3.5,3.5) node {\scriptsize{$4$}};
\draw[gray] (3.5,2.5) node {\scriptsize{$6$}};
\draw[gray] (3.5,1.5) node {\scriptsize{$8$}};
\draw[gray] (3.5,0.5) node {\scriptsize{$10$}};

\draw[gray] (4.5,6.5) node {\tiny{$-5$}};
\draw[gray] (4.5,5.5) node {\tiny{$-3$}};
\draw[gray] (4.5,4.5) node {\tiny{$-1$}};
\draw[gray] (4.5,3.5) node {\scriptsize{$1$}};
\draw[gray] (4.5,2.5) node {\scriptsize{$3$}};
\draw[gray] (4.5,1.5) node {\scriptsize{$5$}};
\draw[gray] (4.5,.5) node {\scriptsize{$7$}};

\draw[gray] (5.5,4.5) node {\tiny{$-4$}};
\draw[gray] (5.5,3.5) node {\tiny{$-2$}};
\draw[gray] (5.5,2.5) node {\scriptsize{$0$}};
\draw[gray] (5.5,1.5) node {\scriptsize{$2$}};
\draw[gray] (5.5,0.5) node {\scriptsize{$4$}};

\draw[gray] (6.5,3.5) node {\tiny{$-5$}};
\draw[gray] (6.5,2.5) node {\tiny{$-3$}};
\draw[gray] (6.5,1.5) node {\tiny{$-1$}};
\draw[gray] (6.5,.5) node {\scriptsize{$1$}};

\draw[gray] (7.5,1.5) node {\tiny{$-4$}};
\draw[gray] (7.5,0.5) node {\tiny{$-2$}};

\end{tikzpicture}
\end{center}
\caption{We apply the gluing algorithm to the periodic paths from Figure \ref{Figure: 4 periodic paths}, using the graph from Figure \ref{Figure-example: graph}. On each step we indicate the gluing point: the first intersection of the Dyck path built so far with the next periodic path. 
\label{Figure: gluing}}
\end{figure}
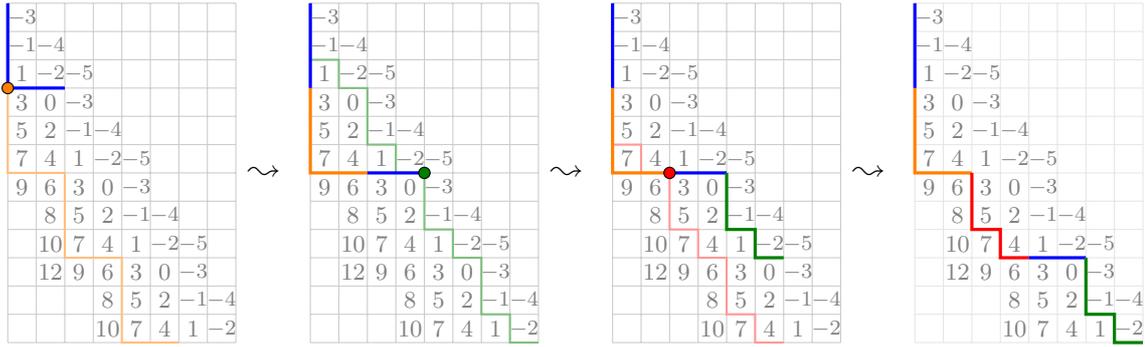

We need to show that this process is invertible. Consider an $(N,M)$-Dyck path $D.$ 

First we will define {\em removal} of intervals.
Let $D$ be a $(kn,km)$-Dyck path. We call $I$ a {\em balanced}
interval of $D$ if it consists of $n+m$ consecutive steps of
$D$ of which $n$ are vertical and $m$ are horizontal.
Using our function notation, this means $I$
is the restriction of $D$ to $[r, r+n+m]$, with $r \in \BZ$ and
$ D(r) = D(r+n+m) + (m,-n)$.
We will say $D'$ is obtained from $D$ by removing a balanced
interval $I$ if it corresponds to the function given by
$${D'}(z) = \begin{cases}
 D (z) & 0 \le z \le r\\
 D (z +(m+n)) + (m, -n)  & r\le z \le (k-1)(n+m)
\end{cases}.$$

\begin{definition}
An interval $I$ of $D$ is called {\it good} if it is of length $n+m,$ balanced,
and its $(n,m)$-periodic extension does not intersect the part of $D$ before $I.$ 
\end{definition}

Our definition of good is motivated by the need to invert the
gluing process. Thus good intervals must have the following properties.
Suppose we remove a good interval $I$ from a $(kn,km)$-Dyck path
$D$ which yields $D'$. If we now glue the periodic extension of $I$
into  $D'$ by taking the first (lowest rightmost) point of $D'$ that
intersects  the periodic path, this should yield the original $D$,
i.e., $D = D' \vee P(I)$.
That good removal inverts the gluing algorithm
 is based on the following Lemmas.

\begin{lemma}
There always exists at least one good interval in a $(kn,km)$-Dyck path
$D$.
\end{lemma}

\begin{proof}
The proof goes in two steps. First, one can show that there always exists at least one balanced interval of $D$ of length $(n+m).$ This is equivalent to showing that the intersection of $D$ with $D$ translated by $n$ down and $m$ to the right, i.e. $D$ intersect $D + (m,-n)$, is non-empty. 
Since $D$ is a Dyck path, it stays weakly below the diagonal
connecting its start with its end. It follows that $D+(m,-n)$
intersects  the vertical line through the start of $D$ (weakly)
below the start of  $D$, and ends (weakly) above $D$. Therefore,
it has to intersect $D$. 

Second, the balanced interval closest to the bottom-right end of $D$ is always good. 
\end{proof}

\begin{lemma}\label{Lemma: extensions of good intervals}
Periodic extensions of the good intervals of $D$ do not intersect each other. 
\end{lemma}

\begin{proof}
Indeed, otherwise the one further away from the bottom-right end of $D$ is not good. 
\end{proof}

\begin{lemma}\label{Lemma: good intervals}
If $I$ and $J$ are good intervals of $D,$ and $D'$ is obtained from $D$ by removing $I,$ then the image of  $J$ in $D'$
 is still a good interval of $D'.$
\end{lemma}

\begin{proof}
Indeed, if the periodic extension of $J$ intersects the part of $D'$ before it, then it also intersects the part of $D$ before it.
\end{proof}

\begin{lemma}
Let $G\in T^d_{n,m}$ and
suppose that the $(N,M)$-Dyck path $D$ was obtained from the periodic paths $P_1,\ldots,P_{d-1}$ according to the gluing algorithm given by $G$. 
The periodic extension $P(I)$ of good intervals $I$ of $D$ agree with the $P_i$ for $i$ the sinks of $G$. 
\end{lemma}

\begin{proof}
In this proof we will use ``before" and ``after" according to
the steps of the gluing algorithm (temporally), and switch to
``higher" or ``lower" to refer to locations of steps or lattice points
of paths.

Let $I$ be a good interval of $D.$ Then, according to Lemma \ref{Lemma: good intervals} either it was glued in on the last step of our algorithm, or it was already a good interval before the last step. In the latter case, the same holds for the second to the last step and so on. Therefore, $I$ was glued in at some point. 

Suppose that $I$ corresponds to the vertex $i$ of $G$, and suppose that there is an edge $i\to j$ in $G.$ Then the interval corresponding to $j$ was glued in after $I$ was. Therefore, either it was glued in in the middle of $I,$ or lower than $I$, in which case after that gluing $I$ is not a good interval any more, because periodic paths $P_i$ and $P_j$ intersect. Contradiction.
\end{proof}

\begin{theorem}
The map $B: T_{n,m}^d\to \YNM$ is a bijection.
\end{theorem}

\begin{proof}
We will induct on $d.$ The case $d=1$ corresponds to the relatively prime case.

 There is some flexibility in the gluing algorithm: if two periodic paths do not intersect each other, then as noted previously it does not matter in which order we glue in intervals of these paths. In particular, we can change the order so that the paths corresponding to the sink vertices $i$ of the graph $G$ are glued in at the last step of the gluing algorithm. Suppose that $G$ has $k$ sink vertices 
$\{i_1, \ldots, i_k\}$ and let $G'\in T_{n,m}^{d-k}$ be the labeled graph obtained from $G$ by removing the sink vertices.
Let also $(s_{i_1},\ldots, s_{i_k})$ be the skeletons corresponding to the sink vertices of $G.$ Note that the following two properties are satisfied:
\begin{enumerate}
\item The skeletons $s_{i_1},\ldots,s_{i_k}$ are pairwise disjoint,
\item Every skeleton corresponding to a sink vertex of the graph $G'$
intersects at least one of the skeletons $s_i, i \not\in \{i_1, \ldots, i_k\}$.
\end{enumerate}
Indeed, if the first property is not satisfied then the corresponding two vertices of $G$ are connected by an edge and cannot both be sinks. 
If the second property is not satisfied, then the corresponding vertex is also a sink of $G,$ which is a contradiction. Conversely, for any $0<k<d,$ any labeled graph $G'\in T_{n,m}^{d-k}$ and a collection of skeletons $s_{i_1},\ldots, s_{i_k}$ satisfying the above two conditions there is a unique graph $G\in T_{n,m}^d,$ such that $s_{i_1},\ldots,s_{i_k}$ are the labels of the sink vertices of $G,$ and $G'$ is obtained from $G$ by removing the sink vertices.

Exactly the same situation happens on the Dyck path side. Let $D\in \YNM$ be an $(N,M)$-Dyck path. Suppose it has $k$ good intervals, and
let  $D'\in Y_{(N-kn,M-km)}$ be the Dyck path obtained from $D$ by removing the good intervals. Let $P_1,\ldots, P_k$ be the periodic extensions of the good intervals of $D.$ The following two properties are satisfied:
\begin{enumerate}
\item The periodic paths $P_1,\ldots,P_k$ are pairwise disjoint,
\item  The periodic extension of any good interval of the Dyck path $D'$ intersects at least one of the paths $P_1,\ldots,P_k.$
\end{enumerate}
Indeed, if the first property is not satisfied then the corresponding two intervals of $D$ cannot both be good. If the second property is not satisfied, then the corresponding good interval of $D'$ is also a good interval of $D,$ which is a contradiction. Conversely, for any $0<k<d,$ any Dyck path $D'\in Y_{(N-kn,M-km)}$ and a collection of $(n,m)$-periodic paths $P_1,\ldots, P_k$ satisfying the above two conditions there is a unique Dyck path $D\in \YNM,$ such that $P_1,\ldots,P_k$ are the periodic extensions of the good intervals of $D,$ and $D'$ is obtained from $D$ by removing all good intervals.

Using Lemma \ref{Lemma: intersections} and induction on $d$ we now can build the inverse map $B^{-1}:\YNM\to T_{n,m}^d.$ See Figure \ref{Figure: reconstructing G} for an example.
\end{proof}


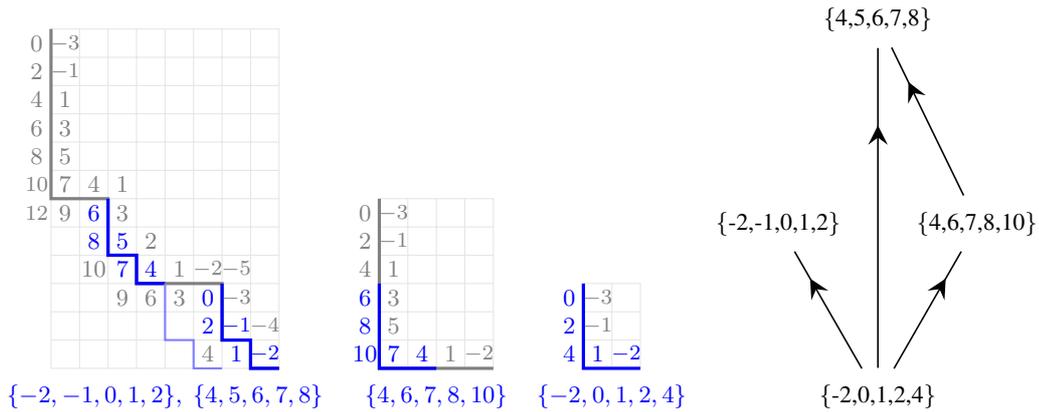
\begin{figure}
\begin{center}
\begin{tikzpicture}[scale=0.375]
\draw [step=1, gray!20!white] (0,0) grid (8,12);


\draw[gray,very thick] (0,12)--(0,6)--(2,6)--(2,4)--(3,4)--(3,3)--(6,3)--(6,1)--(7,1)--(7,0)--(8,0); 

\draw[blue,very thick] (6,3)--(6,1)--(7,1)--(7,0)--(8,0); 
\draw[blue,very thick] (2,6)--(2,4)--(3,4)--(3,3)--(4,3); 

\draw[blue!50!white,thick] (4,3)--(4,1)--(5,1)--(5,0)--(6,0);

\draw[blue] (4,-1) node {\scriptsize{$\{-2,-1,0,1,2\},\ \{4,5,6,7,8\}$}};

\draw[gray] (-0.5,11.5) node {\scriptsize{$0$}};
\draw[gray] (-0.5,10.5) node {\scriptsize{$2$}};
\draw[gray] (-0.5,9.5) node {\scriptsize{$4$}};
\draw[gray] (-0.5,8.5) node {\scriptsize{$6$}};
\draw[gray] (-0.5,7.5) node {\scriptsize{$8$}};
\draw[gray] (-0.5,6.5) node {\tiny{$10$}};
\draw[gray] (-0.5,5.5) node {\tiny{$12$}};

\draw[gray] (.5,11.5) node {\tiny{$-3$}};
\draw[gray] (.5,10.5) node {\tiny{$-1$}};
\draw[gray] (.5,9.5) node {\scriptsize{$1$}};
\draw[gray] (.5,8.5) node {\scriptsize{$3$}};
\draw[gray] (.5,7.5) node {\scriptsize{$5$}};
\draw[gray] (.5,6.5) node {\scriptsize{$7$}};
\draw[gray] (.5,5.5) node {\scriptsize{$9$}};

\draw[gray] (1.5,6.5) node {\scriptsize{$4$}};
\draw[blue] (1.5,5.5) node {\bf\scriptsize{$6$}};
\draw[blue] (1.5,4.5) node {\bf\scriptsize{$8$}};
\draw[gray] (1.5,3.5) node {\scriptsize{$10$}};

\draw[gray] (2.5,6.5) node {\scriptsize{$1$}};
\draw[gray] (2.5,5.5) node {\scriptsize{$3$}};
\draw[blue] (2.5,4.5) node {\bf\scriptsize{$5$}};
\draw[blue] (2.5,3.5) node {\bf\scriptsize{$7$}};
\draw[gray] (2.5,2.5) node {\scriptsize{$9$}};

\draw[gray] (3.5,4.5) node {\scriptsize{$2$}};
\draw[blue] (3.5,3.5) node {\bf\scriptsize{$4$}};
\draw[gray] (3.5,2.5) node {\scriptsize{$6$}};

\draw[gray] (4.5,3.5) node {\scriptsize{$1$}};
\draw[gray] (4.5,2.5) node {\scriptsize{$3$}};

\draw[gray] (5.5,3.5) node {\tiny{$-2$}};
\draw[blue] (5.5,2.5) node {\bf\scriptsize{$0$}};
\draw[blue] (5.5,1.5) node {\bf\scriptsize{$2$}};
\draw[gray] (5.5,0.5) node {\scriptsize{$4$}};

\draw[gray] (6.5,3.5) node {\tiny{$-5$}};
\draw[gray] (6.5,2.5) node {\tiny{$-3$}};
\draw[blue] (6.5,1.5) node {\bf\tiny{$-1$}};
\draw[blue] (6.5,.5) node {\bf\scriptsize{$1$}};

\draw[gray] (7.5,1.5) node {\tiny{$-4$}};
\draw[blue] (7.5,0.5) node {\bf\tiny{$-2$}};

\end{tikzpicture}
\begin{tikzpicture}[scale=0.375]
\draw [step=1, gray!20!white] (0,0) grid (4,6);


\draw[gray,very thick] (0,6)--(0,0)--(4,0); 

\draw[blue,very thick] (0,3)--(0,0)--(2,0); 

\draw[blue] (2,-1) node {\scriptsize{$\{4,6,7,8,10\}$}};

\draw[gray] (-0.5,5.5) node {\scriptsize{$0$}};
\draw[gray] (-0.5,4.5) node {\scriptsize{$2$}};
\draw[gray] (-0.5,3.5) node {\scriptsize{$4$}};
\draw[blue] (-0.5,2.5) node {\bf\scriptsize{$6$}};
\draw[blue] (-0.5,1.5) node {\bf\scriptsize{$8$}};
\draw[blue] (-0.5,.5) node {\bf\scriptsize{$10$}};

\draw[gray] (0.5,5.5) node {\tiny{$-3$}};
\draw[gray] (0.5,4.5) node {\tiny{$-1$}};
\draw[gray] (0.5,3.5) node {\scriptsize{$1$}};
\draw[gray] (0.5,2.5) node {\scriptsize{$3$}};
\draw[gray] (0.5,1.5) node {\scriptsize{$5$}};
\draw[blue] (0.5,.5) node {\bf\scriptsize{$7$}};

\draw[blue] (1.5,.5) node {\bf\scriptsize{$4$}};

\draw[gray] (2.5,.5) node {\scriptsize{$1$}};

\draw[gray] (3.5,.5) node {\tiny{$-2$}};

\end{tikzpicture}
\begin{tikzpicture}[scale=0.375]
\draw [step=1, gray!20!white] (0,0) grid (2,3);

\draw[blue,very thick] (0,3)--(0,0)--(2,0); 

\draw[blue] (1,-1) node {\scriptsize{$\{-2,0,1,2,4\}$}};

\draw[blue] (-0.5,2.5) node {\bf\scriptsize{$0$}};
\draw[blue] (-0.5,1.5) node {\bf\scriptsize{$2$}};
\draw[blue] (-0.5,.5) node {\bf\scriptsize{$4$}};

\draw[gray] (0.5,2.5) node {\tiny{$-3$}};
\draw[gray] (0.5,1.5) node {\tiny{$-1$}};
\draw[blue] (0.5,.5) node {\bf\scriptsize{$1$}};

\draw[blue] (1.5,.5) node {\bf\tiny{$-2$}};

\end{tikzpicture}
\begin{tikzpicture}[
        > = stealth, 
            shorten > = 2pt, 
            shorten < = 2pt, 
            auto,
            semithick 
        ]
\tikzset{myptr/.style={decoration={markings,mark=at position .75 with 
    {\arrow[scale=2,>=stealth]{>}}},postaction={decorate}}}
\tikzset{state/.style={circle,draw,minimum size=6mm, inner sep=0pt}, }

    \node (3) {{\scriptsize \{4,5,6,7,8\}} };
    \node[below right  of=3, xshift=6mm, yshift=-20mm]  (2) 
{\scriptsize \{4,6,7,8,10\}};
    \node [below left  of=3, xshift=-6mm, yshift=-20mm] (1)
{\scriptsize \{-2,-1,0,1,2\}};
    \node [below of=3, yshift=-40mm] (0) {\scriptsize \{-2,0,1,2,4\}};

    \draw[myptr] (2) -- (3);
    \draw[myptr] (0) -- (3);
    \draw[myptr] (0) -- (2);
    \draw[myptr] (0) -- (1);
\end{tikzpicture}

%
%

\end{center}
\caption{On the first step we remove two good intervals and record the corresponding skeletons: $\{-2,-1,0,1,2\}$ and $\{4,5,6,7,8\}.$ On the second step there is only one good interval, with the corresponding skeleton $\{4,6,7,8,10\}.$ Finally, on the last step we are left with a $(3,2)$-Dyck path, which is its own good interval. The corresponding skeleton is $\{-2,0,1,2,4\}.$ On the right we have the reconstructed labeled graph. Note the sinks are the first intervals removed.  Note, that it is isomorphic to the graph in Figure \ref{Figure-example: graph}. 
\label{Figure: reconstructing G}}
\end{figure}

We conclude that since the maps $A:\MNM/\mathord\sim\to T_{n,m}^d$ and $B:T_{n,m}^d\to\YNM$ are bijections, the map $\CD=B\circ A:M_{N,M}/\mathord\sim\to \YNM$ is a bijection as well.


\begin{theorem}
The sweep map $\zeta:\YNM\to \YNM$ factorizes according to Figure
\ref{commutative triangle 2}:
$$
\zeta=\CG\circ \CD^{-1}
$$
for all positive $N, M$.
\end{theorem}

\begin{proof}
Let $D\in \YNM$ be a Dyck path, and let $\Delta\in \MNM$ be a minimal representative of the equivalence class $\CD^{-1}(D)\subset\MNM.$ Similar to the $d=1$ case, the steps of the path $D$ correspond to the elements of the skeleton $S$ of $\Delta.$ However, the correspondence is a bit trickier. Since the $\Delta$ is a minimal representative, the rank of a step of $D$ equals  $\lfloor\frac{x}{d}\rfloor,$ where $x$ is the corresponding element of the skeleton of $\Delta$. However, according to the gluing algorithm, if $x,y\in S$ are two elements of the skeleton of $\Delta,$ such that $x<y$ and $\lfloor\frac{x}{d}\rfloor=\lfloor\frac{y}{d}\rfloor,$ then the step corresponding to $x$ is glued in lower than the step corresponding to $y.$ In turn, that implies that the step in $D$ corresponding to $x$ appears higher than the step corresponding to $y,$ which matches with the ``tie breaking'' adjustment in the construction of the sweep map in the non relatively prime case (see Example \ref{Example: zeta nonrp}). 
\end{proof}

The following proposition gives a simple interpretation of the $\area$ statistic for rational Dyck paths in terms of $(N,M)$--invariant subsets.

\begin{proposition}
Let $\Delta\in \MNM$. Then 
$$
\area(\CD(\Delta))=\min_{\Delta'\sim \Delta}\gap(\Delta'),
$$
where, as above, $\gap(\Delta')=|\BZ_{\ge 0}\setminus \Delta'|.$
\end{proposition}
\begin{proof}
As before, let $\Delta_0,\ldots,\Delta_{d-1}$ be the $d$-tuple of $(n,m)$-invariant subsets defined by 
$$
\Delta_r=\left[\left(\Delta\cap (d\mathbb Z+r)\right)-r\right]/d.
$$ 
Let also $m_r=\min\Delta_r.$ Then
$$
\gap(\Delta)=\sum_{r}\gap(\Delta_r)=\sum_r \left[m_r+\gap(\Delta_r-m_r)\right].
$$
Note that $(\Delta_r-m_r)\in \Mnm$ and, in particular, $\gap(\Delta_r-m_r)=\area\left[\CD(\Delta_r-m_r)\right],$ because in the relatively prime case the $\area$ statistic counts the boxes whose ranks are exactly the gaps, and each gap is counted exactly once. It follows that to obtain the $\Delta'$ with  minimal $\gap$ over the equivalence class of $\Delta$ one should consider an invariant subset with the minimal $\sum_r m_r,$ which is equivalent to considering one of the minimal representatives. Therefore, it is sufficient to prove that if $\Delta$ is a minimal representative, then $\area(\CD(\Delta))=\gap(\Delta)$.
Let us compute the area between $\CD(\Delta)$ and the diagonal in the $(N,M)$ rectangle $\RNM$. It consists of the areas between the Dyck paths (possibly shifted and disconnected) for the $(n,m)$-invariant subsets $\Delta_r$
and their local diagonals, and the parallelograms between these small diagonals and the big diagonal. Since $\Delta$ is a minimal representative, the smallest rank of a box that fits under the local diagonal corresponding to $\Delta_r$ is $m_r.$ Therefore, such a parallelogram contains the boxes with all possible ranks between $0$ and $m_r$, each rank appearing exactly once. Therefore $\area(\CD(\Delta))=\gap(\Delta)$.
\end{proof}

\subsection{Example: $k$-Catalan arrangement}
\label{sec-catalan}
Let us describe the equivalence relation in the case $M=kN$, $k\in \BZ_{>0}$. In this case, $d=N$, $n=1$ and $m=k$. 
A module is $(N,M)$-invariant if and only if it is $N$-invariant, and therefore has the form
$$
\Delta(k_0,\ldots,k_{N-1})=\{k_i+Nj\ :\ i=0,\ldots,N-1, j\ge 0\},
$$
where $k_i$ is an arbitrary integer with remainder $i$ modulo $N$. 
To be $0$-normalized we further require $k_0=0$ and $k_i \ge 0$.  Now 
$\Delta_i=\{k_i+Nj:j\ge 0\}$, so the skeleton $S_i$ has a unique $N$-generator $k_i$ and has $k$ $M$-cogenerators
$k_i-N,\ldots,k_i-kN$. Therefore the $i$-th skeleton of $\Delta$ has the form
$$
S_i=\{k_i,k_i-N,\ldots,k_i-kN\}.
$$

Recall that the $k$-Catalan arrangement in $\BR^{N}$ is defined by the equations $x_i-x_j=s$ where $i<j$ and
$s$ runs through $\{-k,\ldots,k\}$, and the $k$-Shi arrangement is defined by the same equations with $s\in \{-(k-1),\ldots,k\}$.
We will call the connected components of their complements $k$-Catalan and $k$-Shi regions, respectively.
Clearly, in the dominant cone where $x_1<\ldots<x_N$ the arrangements agree and it is known that the number of 
dominant $k$-Shi regions is equal to the $n$th Fuss-Catalan number
$$
c_N(k):=\frac{((k+1)N)!}{(kN+1)!N!},
$$
which is also equal to the number of Dyck paths in the $N\times kN$ rectangle. Since the $k$-Catalan arrangement is $S_N$-invariant, the total number of $k$-Catalan regions equals $N!c_N(k)$.

If we pass to $V= \BR^N/\Rspan(1,1,\ldots,1)$,
the connected components of the complement of the hyperplane
arrangement $\{x_i-x_j=s | s \in \BZ \}$ are called alcoves.
Observe that while these regions are unbounded in $\BR^N$,
in $V$ they are bounded and each alcove has centroid
of the form $(\frac {a_1}{N}, \cdots, \frac {a_N}{N} )$
with $a_i \in \BZ$ and  $\{ a_i \bmod N \}$ distinct.
We will always take our representative of 
$(\frac {a_1}{N}, \cdots, \frac {a_N}{N} ) + \Rspan(1,1,\ldots,1)$
to be such that $\min\{a_i\} = 0$. This is compatible
with taking $\Delta$ to be $0$-normalized. 
(Note that in the literature one often normalizes to be
``balanced,"  so  that
$\sum a_i = 0$ and  $\sum k_i = \binom{N+1}{2}$.) 

Note further that to each $\Delta(k_0, \ldots, k_{N-1})$ we can
associate the alcove that has centroid
$p_{\Delta} = (\frac{k_{0}}{N}, \cdots, \frac{k_{N-1}}{N})$.
Since $\Delta$ is independent of the order of the $k_i$,
we could just as easily associate to it the alcove
in the dominant cone $x_1 < \cdots < x_N$
that has centroid $p^+_{\Delta} = (\frac{k_{\sigma(0)}}{N}, \cdots,
\frac{k_{\sigma(N-1)}}{N})$, where $\sigma \in \Perm\{0, 1, \ldots,
N-1\}$ is chosen so that $k_{\sigma(i)} < k_{\sigma(i+1)}$.

\begin{proposition}
The set of integral acceptable shifts for $\Delta$ for which
the shifting is by distinct integers $\bmod N$  is in  bijection
with the set of alcoves that are in the
same $k$-Catalan region as $p_{\Delta}$.
\end{proposition}

\begin{proof}
Indeed, the shifting $(a_0,\ldots,a_{N-1})$ is acceptable if and only
if for all $i$ and $j$ the  order of (colored) points in the sets $S_i
\cup S_j$ and $S_i+a_i \cup S_j+a_j$ is the same.
As the order of $x$ and $y$ (i.e., whether $x<y$) is determined
by whether $x-y<0$ and since our shifting is by integers distinct
$\bmod N$, it suffices to consider the signs of such differences.
 More algebraically, for all pairs
$(x=k_i-tN, y=k_j-t'N)\in S_i \times S_j$ the sign of
$x-y$ and the sign of $(x+a_i)-(y+a_j)$ is the same.
The sign  the $x-y$
is determined by the sign of $k_i/N-k_j/N-(t-t')$, so we require that
the points $(k_0/N,\ldots,k_{N-1}/N)$ and
$((k_0+a_0)/N,\ldots,(k_{N-1}-a_{N-1})/N)$ are on the same side of the
hyperplane $x_i-x_j=t-t'$.  (This is still true if we look at points
whose coordinates
are sorted to lie in the dominant cone, i.e., the alcoves
in the same region as $p^+_\Delta$.)
It remains to notice that possible values of $t-t'$ run between $-k$ and $k$. 
\end{proof}


We conclude that the set of equivalence classes of $(N,kN)$--invariant
subsets is in bijection with the set of dominant $k$-Catalan (or,
equivalently, dominant $k$-Shi) regions, and both sets have $c_N(k)$
elements. Therefore our main construction provides yet another
bijection [FV10] between dominant $k$-Shi regions and Dyck paths in
$N\times kN$ rectangle. 

\section{Relation to knot invariants}

In this section we prove Theorem \ref{EH comparison}. We will use the following result:

\begin{theorem}(\cite[Theorem 1.9]{EH})
\label{th: EH}
The Poincar\'e series of the $(a=0)$ part of the Khovanov-Rozansky homology of the $(n,n)$ torus link equals
$$
F_n(q,t)=\sum_{a=(a_1,\ldots,a_n)\in \BZ_{\ge 0}^{n}}q^{\sum a_i}t^{d(a)},
$$
where $d(a)=|\{i<j\ :\ a_i=a_j\ \text{or}\ a_j=a_i+1\}|$.
\end{theorem}

\begin{lemma}
\label{lem:cyclic shift}
One has
$$
(1-q)F_n(q,t)=\sum_{a=(a_1,\ldots,a_{n-1},0)\in \BZ_{\ge 0}^{n}}q^{\sum a_i}t^{d(a)}.
$$
\end{lemma}
\begin{proof}
Let us define the cyclic shift operator $\pi:(a_1,\ldots,a_n)\mapsto (a_n-1,a_1,\ldots,a_{n-1}),$
which is well defined if $a_n>0$. By applying $\pi$ repeatedly, we can transform a given tuple $a$ to a tuple with $a_n=0$.
Clearly, $\sum \pi(a)=\sum a_i-1$ and one can check that
$d(\pi(a))=d(a)$. Therefore:
$$
F_n(q,t)=\sum_{k\ge 0}\sum_{a=(a_1,\ldots,a_{n-1},0)\in \BZ_{\ge 0}^{n}}q^{k+\sum a_i}t^{d(a)}=\frac{1}{1-q}\sum_{a=(a_1,\ldots,a_{n-1},0)\in \BZ_{\ge 0}^{n}}q^{\sum a_i}t^{d(a)}.
$$
\end{proof}

We are ready to prove Theorem \ref{EH comparison}.

\begin{proof}[Proof of Theorem \ref{EH comparison}]
By Lemma \ref{lem:cyclic shift} and Theorem \ref{th: EH} we need to prove the identity
$$
C_{n,n}(q,t)=\sum_{a=(a_1,\ldots,a_{n-1},0)\in \BZ_{\ge 0}^{n}}q^{\sum a_i}t^{d(a)}.
$$
A subset $\Delta\subset \BZ_{\ge 0}$ is $(n,n)$--invariant if and only if it is $n$--invariant. In remainder $i$ it has an $n$-generator $x_i=i+na_i$ and an $n$-cogenerator $y_i=i+na_i-n$, for some $a_i\ge 0$. It is $0$-normalized if $a_n=0$. It is easy to check that 
$\gap(\Delta)=\sum a_i$. Now
$$
\dinv(\Delta)=\binom{n}{2}-\sharp\{i,j:y_j>x_i\}=\binom{n}{2}-\sharp\{i<j:a_j>a_i+1\}-\sharp\{i>j:a_j>a_i\}=d(a).
$$
\end{proof}

\begin{example}
Let us compute $C_{2,2}(q,t)$. All 2-invariant 0-normalized subsets have the form 
$$
\Delta_k=\{0,2,\ldots,2k,2k+1,2k+2,\ldots\}.
$$
Clearly, $\gap(\Delta_k)=k$ and 
$$
\dinv(\Delta_k)=\begin{cases}
1,\ \text{if}\ k=0,\\
0,\ \text{if}\ k>0.\\
\end{cases}
$$
Therefore
$$
C_{2,2}(q,t)=\sum_{k=0}^{\infty}q^{\gap(\Delta_k)}t^{\dinv(\Delta_k)}=t+\frac{q}{1-q}=\frac{q+t-qt}{1-q}.
$$
Note that $c_{2,2}(q,t)=q+t$.
\end{example}

\end{document}